\documentclass[11pt,a4paper]{article}

\usepackage{amssymb} 
\usepackage{amsthm} 
\usepackage{booktabs} 
\usepackage[a4paper, total={6.3in, 8in}]{geometry}
\usepackage{enumerate} 
\usepackage[acronym,nomain]{glossaries} 
\usepackage{hyperref}
\usepackage{mathrsfs}  
\usepackage{mathtools} 
\usepackage{upgreek} 

\usepackage{cleveref} 

\theoremstyle{plain}
\newtheorem{theorem}{Theorem}[section]
\crefname{theorem}{theorem}{theorems}
\Crefname{theorem}{Theorem}{Theorems}

\newtheorem{corollary}[theorem]{Corollary}
\crefname{corollary}{corollary}{corollaries}
\Crefname{corollary}{Corollary}{Corollaries}

\theoremstyle{definition}
\newtheorem{definition}[theorem]{Definition}
\crefname{definition}{definition}{definitions}
\Crefname{definition}{Definition}{Definitions}

\theoremstyle{remark}
\newtheorem{remark}[theorem]{Remark}
\crefname{remark}{remark}{remarks}
\Crefname{remark}{Remark}{Remarks}

\newcommand{\amplitude}{\alpha}
\newcommand{\arrhenius}{\beta}
\newcommand{\Bh}{B_h}
\newcommand{\calV}{\mathcal{V}}
\newcommand{\calW}{\mathcal{W}}
\newcommand{\colsp}{\mathrm{colsp}}
\newcommand{\deltaMax}{\delta_{\mathrm{max}}}
\newcommand{\diag}{\mathrm{diag}}
\newcommand{\diff}{\mathrm{d}}
\newcommand{\diffusion}{d}
\newcommand{\diss}{\mathcal{D}}
\newcommand{\Eh}{E_h}
\newcommand{\error}{\varepsilon}
\newcommand{\F}{F}
\newcommand{\FRed}{\tilde{\F}}
\newcommand{\fulldim}{n}
\newcommand{\gr}{g_\mathrm{r}}
\newcommand{\ham}{\mathcal{H}}
\newcommand{\hamAux}{\ham_{\mathrm{s}}}
\newcommand{\hamh}{\ham_h}
\newcommand{\Hom}{\mathcal{L}}
\newcommand{\Jh}{J_h}
\newcommand{\maxSingularVal}{\sigma_{\mathrm{max}}}
\newcommand{\minSingularVal}{\sigma_{\mathrm{min}}}
\newcommand{\modeDim}{d_{\upphi}}
\newcommand{\modes}{\phi}
\newcommand{\MPOD}{M}
\newcommand{\nMesh}{N}
\newcommand{\NN}{\mathbb{N}}
\newcommand{\preExpFac}{\zeta}
\newcommand{\rDim}{r}
\newcommand{\reactionRate}{v}
\newcommand{\reactionRateMat}{V}
\newcommand{\redB}{\tilde{B}}
\newcommand{\redDiss}{\tilde{\diss}}
\newcommand{\redE}{\tilde{E}}
\newcommand{\redHam}{\tilde{\ham}}
\newcommand{\redHamAux}{\redHam_{\mathrm{s}}}
\newcommand{\redHamAuxDot}{\dot{\redHam}_{\mathrm{s}}}
\newcommand{\redJ}{\tilde{J}}
\newcommand{\redK}{\tilde{K}}
\newcommand{\redQ}{\tilde{Q}}
\newcommand{\redR}{\tilde{R}}
\newcommand{\redy}{\tilde{y}}
\newcommand{\res}{\mathscr{R}}
\newcommand{\Rh}{R_h}
\newcommand{\rph}{\mathfrak{r}}
\newcommand{\RR}{\mathbb{R}}
\newcommand{\shift}{p}
\newcommand{\sol}{x}
\newcommand{\solh}{{\sol}_h}

\newcommand{\spaceVar}{\xi}

\newcommand{\supplyRate}{\mathcal{S}}
\newcommand{\tend}{t_{\mathrm{end}}}
\newcommand{\timeInterval}{\mathbb{I}}
\newcommand{\Text}{\mathcal{T}_{\mathrm{ext}}}
\newcommand{\Tper}{\mathcal{T}_{\mathrm{per}}}
\newcommand{\tridiag}{\mathrm{tridiag}}
\newcommand{\Vr}{V_\mathrm{r}}
\newcommand{\Vs}{V_\mathrm{s}}
\newcommand{\Wr}{W_\mathrm{r}}
\newcommand{\zph}{z}
\newcommand{\zset}{\lbrace 0\rbrace} 

\newcommand{\redrph}{\tilde{\rph}}
\newcommand{\redState}{\tilde{\sol}}
\newcommand{\redSupplyRate}{\tilde{\supplyRate}}

\newcommand{\integral}[4]{\int\limits_{#1}^{#2} #3\,\diff #4}
\newcommand{\ip}[1]{\left\langle #1\right\rangle}
\newcommand{\LTwo}[1]{L^{2}\left(#1\right)}
\newcommand{\norm}[1]{\left\lVert #1\right\rVert}

\DeclareMathOperator*{\argmin}{arg\,min}

\newacronym{dae}{DAE}{differential--algebraic equation}
\newacronym{fem}{FEM}{finite element method}
\newacronym{fom}{FOM}{full-order model}
\newacronym{mor}{MOR}{model order reduction}
\newacronym{ode}{ODE}{ordinary differential equation}
\newacronym{pde}{PDE}{partial differential equation}
\newacronym{ph}{pH}{port-Hamiltonian}
\newacronym{rom}{ROM}{reduced-order model}

\begin{document}
\title{Structure-Preserving Model Reduction for Port-Hamiltonian Systems Based on a Special Class of Nonlinear Approximation Ansatzes}
\author{Philipp Schulze\,\thanks{Department of Mathematics, Chemnitz University of Technology, Chemnitz, Germany,
\texttt{philipp.schulze@math.tu-chemnitz.de}.}}

\maketitle            

\begin{abstract}
We discuss structure-preserving model order reduction for port-Hamiltonian systems based on an approximation of the full-order state by a linear combination of ansatz functions which depend themselves on the state of the reduced-order model. 
In recent years, such nonlinear approximation ansatzes have gained more and more attention especially due to their effectiveness in the context of model reduction for transport-dominated systems which are challenging for classical linear model reduction techniques. 
We demonstrate that port-Hamiltonian reduced-order models can often be obtained by a residual minimization approach where a special weighted norm is used for the residual. 
Moreover, we discuss sufficient conditions for the resulting reduced-order models to be stable. 
Finally, the methodology is illustrated by means of two transport-dominated numerical test cases, where the ansatz functions are determined based on snapshot data of the full-order state.
\vskip .3truecm

\vskip .3truecm

{\bf AMS(MOS) subject classification:} 34D20, 35Q49, 80A25, 93D30


\end{abstract}

\section{Introduction}

In applications where multi-query evaluations of a computational model are required, such as optimization, control, or uncertainty quantification, \gls{mor} techniques provide a powerful tool for accelerating the overall procedure, while maintaining an acceptable model accuracy.
For a general overview of such methods, we refer to \cite{Ant05b,BenOCW17,HesRS16,QuaMN16,SchVR08}.
While most \gls{mor} techniques are based on a linear approximation of the \gls{fom} state, recently, methods based on nonlinear approximation ansatzes have received more and more attention, cf.~\cite{BarF22,BlaSU20,CagMS19,FreDM21,KimCWZ22,LeeC19} and the references therein.
One reason for the great interest in the latter class of methods is that \gls{mor} methods based on linear approximation ansatzes are often inadequate for an effective reduction of transport-dominated systems, see for instance \cite{OhlR16}.

While nonlinear model reduction techniques may lead to very low-dimensional and still accurate \glspl{rom}, they do in general not guarantee that the \gls{rom} inherits important system properties of the corresponding \gls{fom}, such as stability.
This may lead to unphysical behavior of the \gls{rom} or to approximation errors which grow exponentially with respect to time.
In fact, the same issue also applies to most linear \gls{mor} methods and, thus, several \gls{mor} techniques have been proposed which preserve certain qualitative properties of the \gls{fom}, cf.~\cite{Ant05a,BenS17,BreU22,CasLE12,CheS21,MonTC13,Pul19,Sor05}.
In this paper, we are especially interested in structure-preserving \gls{mor} methods for \gls{ph} systems, since these come with many desirable properties, cf.~\cite{Kot19,MehU23,SchJ14} for a general overview.
For instance, a \gls{ph} structure implies passivity and often also stability of the dynamical system.
In addition, \gls{ph} structures are closed under power-preserving interconnection, which makes them especially attractive for control purposes \cite{DuiMSB09,OrtSCA08,Sch20,SchPFWM21} and for modeling networks \cite{AltMU21,FiaZOSS13}. 
Besides, since the energy plays the central role in the \gls{ph} framework, it is suitable for a wide range of applications \cite{AltS17,BanSAZISW21,BruAPM19a,BruAPM19b,GerHRS21,HoaCJL11,MacM04,MorLMRY21,RamMS13,RasCSS21,WanMS18,WarBST21} as well as for the coupling of different physical domains \cite{CarMP17,FalH17,VosS14,VuLM16,ZhoWHLW21}.

The literature on structure-preserving \gls{mor} techniques for \gls{ph} systems is mainly focused on linear time-invariant systems and linear approximation ansatzes.
Common techniques are based on balanced truncation \cite{BorSF21,BreMS22,HarVS10,PolS12} or transfer function interpolation \cite{EggKLMM18,GifWPL14,GugPBS12,IonA13a,WolLEK10}.
Aside from these projection-based techniques, also direct optimization approaches have been proposed for obtaining a port-Hamiltonian \gls{rom}, cf.~\cite{Sat18,SchV21}.
In \cite{ChaBG16}, the authors present a structure-preserving \gls{mor} approach for nonlinear \gls{ph} systems using a linear approximation ansatz based on the proper orthogonal decomposition or the iterative rational Krylov algorithm presented in \cite{GugAB08}.
In addition, they propose a structure-preserving variant of the discrete empirical interpolation method \cite{ChaS10}, which enables an efficient evaluation of the \gls{rom}.

While the contributions mentioned in the previous paragraph are based on linear approximation ansatzes, we present in this paper a structure-preserving \gls{mor} approach for a special class of nonlinear \gls{ph} systems based on a special class of nonlinear approximation ansatzes, cf.~\cref{sec:problemSetting} for the details.
The main contributions are listed in the following.

\begin{itemize}
	\item In \Cref{thm:linearPH_linearMOR} we show that most classical structure-preserving \gls{mor} techniques for linear \gls{ph} systems are optimal in the sense that the derivative of the \gls{rom} state minimizes a weighted norm of the residual.
		Moreover, in \Cref{rem:motivationForWeightedResNorm} we present another motivation for using this weighted norm by providing a relation to residual-based error bounds.
	\item We present a structure-preserving \gls{mor} approach using a special class of nonlinear approximation ansatzes which is especially relevant in the context of transport-dominated systems.
		Especially, this approach is not limited to linear systems, but may be applied to a special class of nonlinear \gls{ph} systems, cf.~\Cref{thm:nonlinearPHDAE_Q_factorizableMOR}.
	\item We identify a subclass of nonlinear ansatzes which allows to obtain \glspl{rom} which are at the same time \gls{ph} and optimal in terms of residual minimization, provided that the corresponding \gls{fom} is a linear \gls{ph} system, cf.~\Cref{thm:LTVPH_separableMOR}.
	\item We provide sufficient conditions which ensure that the state equations of the \gls{ph} reduced-order models with vanishing input signal are stable, cf.~\Cref{cor:nonlinearPHDAE_Q_factorizableMOR_stability}, \Cref{rem:nonlinearPHDAE_Q_factorizableMOR_stability_specialCases}, and \Cref{cor:nonlinearPHDAE_Q_separableMOR_stability}.
	\item We present a new \gls{ph} representation of a wildland fire model which has been considered for example in \cite{BlaSU21b,ManBBCDKV08}, cf.~\cref{sec:wildlandFire}.
\end{itemize}

The remainder of this manuscript is structured as follows.
In \cref{sec:problemSetting} we briefly present the mathematical setting considered in this paper by introducing the treated classes of \gls{ph} systems and approximation ansatzes.
We proceed by providing some preliminary definitions and results in \cref{sec:preliminaries} with particular focus on differential equations, \gls{ph} systems, and \gls{mor}.
The main results are presented in \cref{sec:pHMOR} where we especially demonstrate how to achieve structure-preserving \gls{mor} for \gls{ph} systems based on linear and special nonlinear approximation ansatzes.
Finally, we illustrate the theoretical results by means of two numerical test cases in \cref{sec:numerics} and provide a summary and an outlook in \cref{sec:conclusion}.

\paragraph*{Notation}

The set of real numbers is denoted with $\RR$ and we use $\RR^{m, n}$ for the set of $m\times n$ matrices with real-valued entries.
In particular, the $n\times n$ identity matrix is denoted by $I_n$ and the transpose of a matrix $A$ by $A^\top$.
Furthermore, to indicate that a matrix $A\in\RR^{m,m}$ is positive (semi-)definite, we use the notation $A>0$ ($A\ge 0$).
In addition, $\colsp(A)$, $\maxSingularVal(A)$, and $\minSingularVal(A)$ denote the column space, the maximum, and the minimum singular value of a matrix $A$, respectively.
Besides, we use the abbreviations
\begin{equation*}
	\diag(a_1,\ldots,a_n) \vcentcolon= 
	\begin{bmatrix}
		a_1 & 			& \\
			& \ddots 	& \\
			&			& a_n 
	\end{bmatrix}
	,\quad \tridiag_n(a,b,c) \vcentcolon=
	\begin{bmatrix}
		b 	& c & 			& 			& 			& 	& \\
		a 	& b & c 			& 			& 			& 	& \\
		 	& a & b 			& c 			& 			& 	& \\
		 	& 	& \ddots 	& \ddots 	& \ddots 	& 	& \\
		 	& 	& 			& a 			& b 			& c & \\
		 	& 	& 			&  			& a 			& b & c\\
		 	& 	& 			&  			&  			& a & b
	\end{bmatrix}
	\in\RR^{n,n}
\end{equation*}
for diagonal matrices and tridiagonal Toeplitz matrices, respectively.
For column vectors, we abbreviate $\RR^{m,1}$ as $\RR^m$ and we write $\norm{\cdot}$ for the Euclidean norm on $\RR^m$.
Given an interval $\Omega\vcentcolon=(a,b)$ with $a\in\RR$ and $b\in\RR_{>a}$, we denote the Hilbert space of square-integrable functions over $\Omega$ as $\LTwo{\Omega}$ and use $\ip{\cdot,\cdot}_{\LTwo{\Omega}}$ for the corresponding inner product.
Furthermore, for the Sobolev space of square-integrable functions with square-integrable weak derivative, we use the notation $H^1(\Omega)$.
The spaces of continuous and continuously differentiable functions from a suitable subset $U\subseteq \RR^m$ to $\RR^n$ are denoted with $C(U,\RR^n)$ and $C^1(U,\RR^n)$, respectively.
Finally, for a function $f$ depending on multiple variables $x_1,\ldots,x_m$, we use the short-hand notation $\partial_{x_i}f\vcentcolon=\frac{\partial f}{\partial x_i}$ for the partial derivative of $f$ with respect to $x_i$ for $i\in\lbrace 1,\ldots,m\rbrace$.

\section{Problem Setting}
\label{sec:problemSetting}
In this paper, we consider the problem of structure-preserving \gls{mor} for \gls{ph} systems of the form
\begin{subequations}
	\label{eq:nonlinearPHDAE_Q}
	\begin{align}
		\label{eq:nonlinearPHDAE_Q_stateEq}
		E(t,\sol(t))\dot{\sol}(t)+\rph(t,\sol(t)) 	&= (J(t,\sol(t))-R(t,\sol(t)))Q(t,\sol(t))\sol(t)+B(t,\sol(t))u(t),\\
		y(t) 												&= B(t,\sol(t))^\top Q(t,\sol(t))\sol(t)
	\end{align}
\end{subequations}
for all $t\in\timeInterval \vcentcolon= [t_0,\tend]$, with $t_0\in\RR_{\ge 0}$, $\tend\in\RR_{>t_0}$, state $\sol\colon \timeInterval\to \RR^{\fulldim}$, input port $u\colon \RR_{\ge 0}\to \RR^m$, output port $y\colon \timeInterval\to \RR^m$, and coefficients $E,J,R,Q\in C(\RR_{\ge 0}\times \RR^{\fulldim},\RR^{\fulldim,\fulldim})$, $\rph\in C(\RR_{\ge 0}\times \RR^{\fulldim},\RR^{\fulldim})$, and $B\in C(\RR_{\ge 0}\times \RR^{\fulldim},\RR^{\fulldim,m})$.
Associated with \eqref{eq:nonlinearPHDAE_Q} we consider a Hamiltonian $\ham\in C^1(\RR_{\ge 0}\times \RR^n)$ and the coefficient functions are required to satisfy
\begin{equation}
	\label{eq:nonlinearPHDAEProperties_Q_new}
	\begin{aligned}
		J(t,\sol) &= -J(t,\sol)^\top, &\nabla_{\sol}\ham(t,\sol) &= E(t,\sol)^\top Q(t,\sol)\sol,\\ 
		R(t,\sol) &= R(t,\sol)^\top	\ge 0,\quad &\partial_t\ham(t,\sol) &= \rph(t,\sol)^\top Q(t,\sol)\sol
	\end{aligned}
\end{equation}
for all $(t,\sol)\in \RR_{\ge 0}\times \RR^{\fulldim}$.
The \gls{ph} structure given by \eqref{eq:nonlinearPHDAE_Q}--\eqref{eq:nonlinearPHDAEProperties_Q_new} is a special case of the one introduced in \cite{MehM19}, where the authors consider in addition a feedthrough term, a possibly non-square $E$ matrix, and a general function $\zph\in C(\RR_{\ge 0}\times \RR^{\fulldim},\RR^{\fulldim})$ instead of $\zph$ being defined as $\zph(t,\sol) \vcentcolon= Q(t,\sol)\sol$.
While the framework presented in \cref{sec:pHMOR} may be extended to systems with non-vanishing feedthrough term and some of the results also to the case of a non-square $E$ matrix, the assumption on the special structure of $\zph$ is essential.

In \cite{MehM19} it is illustrated that the temporal change of the Hamiltonian along solution trajectories of \eqref{eq:nonlinearPHDAE_Q_stateEq} is bounded from above by the supplied power $y^\top u$.
Under additional assumptions on the Hamiltonian, this leads to a stability result for the unforced system \eqref{eq:nonlinearPHDAE_Q_stateEq} with $u=0$, cf.~\cref{sec:pH}.
In order to obtain \glspl{rom} which inherit these properties, it is desirable to develop \gls{mor} schemes which preserve the structure \eqref{eq:nonlinearPHDAE_Q}--\eqref{eq:nonlinearPHDAEProperties_Q_new}.

While structure-preserving \gls{mor} based on linear approximation ansatzes has been investigated since at least a decade, the aim of this paper is to present structure-preserving \gls{mor} schemes based on nonlinear approximation ansatzes.
While we do not address the most general class of nonlinear ansatzes in this paper, we consider ansatzes of the form
\begin{equation}
	\label{eq:timeDependentFactorizableAnsatz}
	\sol(t) \approx \Vr(t,\redState(t))\redState(t)
\end{equation}
with $\Vr\colon \RR_{\ge 0}\times \RR^{\rDim}\to \RR^{\fulldim,\rDim}$ and \gls{rom} state $\redState\colon\timeInterval\to \RR^{\rDim}$.
Such ansatzes are especially relevant in the context of \gls{mor} for transport-dominated systems, see \cref{sec:factorizableMOR,sec:separableMOR,sec:numerics} for some examples.

Based on the considered class of approximation ansatzes, the task considered in this work is to introduce a projection-based \gls{mor} framework for constructing port-Hamiltonian \glspl{rom} of the form
\begin{subequations}
	\label{eq:nonlinearPHDAE_Q_factorizableMOR}
	\begin{align}
		\label{eq:nonlinearPHDAE_Q_factorizableMOR_stateEq}
		\redE(t,\redState(t))\dot{\redState}(t)+\redrph(t,\redState(t)) 	&= (\redJ(t,\redState(t))-\redR(t,\redState(t)))\redQ(t,\redState(t))\redState(t)+\redB(t,\redState(t)) u(t),\\
		\redy(t)																				&= \redB(t,\redState(t))^\top\redQ(t,\redState(t))\redState(t)
	\end{align}
\end{subequations}
for all $t\in\timeInterval$, with $\redE,\redJ,\redR,\redQ\colon\RR_{\ge 0}\times \RR^{\rDim}\to\RR^{\rDim,\rDim}$, $\redrph\colon \RR_{\ge 0}\times \RR^{\rDim}\to\RR^{\rDim}$, $\redB\colon \RR_{\ge 0}\times \RR^{\rDim}\to\RR^{\rDim,m}$, and $\redy\colon \timeInterval\to\RR^m$.
Here, the \gls{rom} coefficient functions are assumed to satisfy conditions analogous to \eqref{eq:nonlinearPHDAEProperties_Q_new} with \gls{rom} Hamiltonian $\redHam\colon \RR_{\ge 0}\times\RR^{\rDim}\to\RR$ defined via
\begin{equation}
	\label{eq:redHam}
	\redHam(t,\redState) \vcentcolon= \ham(t,\Vr(t,\redState)\redState).
\end{equation}
In addition to the structure preservation itself, we are also interested in deriving conditions which ensure that the state equation of the unforced \gls{rom} with $u=0$ has a stable equilibrium point in the origin.

\begin{remark}[Preservation of algebraic constraints]
	\label{rem:algebraicConstraints}
	We emphasize that the general \gls{ph} structure \eqref{eq:nonlinearPHDAE_Q}--\eqref{eq:nonlinearPHDAEProperties_Q_new} includes the case of a singular $E$ matrix.
	In this context one is often also interested in preserving the algebraic constraints of the system, see for instance \cite{BenS17} and the references therein.
	However, in the following we only focus on preserving the \gls{ph} structure, while we refer to \cite{BeaGM22,EggKLMM18,HauMM19b,MehU23} for contributions focusing on structure-preserving \gls{mor} for port-Hamiltonian \gls{dae} systems.
\end{remark}

\section{Preliminaries}
\label{sec:preliminaries}

In this section, we present some preliminary definitions and results needed for the following sections.
We start by addressing general differential equation systems, existence and uniqueness of their solutions, and stability of equilibrium points in \cref{sec:DEsAndStability}.
Port-Hamiltonian systems and some of their properties are treated in \cref{sec:pH}, while \cref{sec:projectionMOR} is devoted to \gls{mor} schemes based on different projection techniques.

\subsection{Differential Equations and Stability}
\label{sec:DEsAndStability}

Throughout this paper, we consider finite-dimensional systems of the form
\begin{subequations}
	\label{eq:finiteDimensionalODE_IVP}
	\begin{align}
		\label{eq:finiteDimensionalODE}
		E(t,\sol(t))\dot{\sol}(t) 	&= \F(t,\sol(t))\quad\text{for all }t\in\timeInterval,\\ 
		\label{eq:finiteDimensionalODE_IC}
		\sol(t_0)						&=\sol_0
	\end{align}
\end{subequations}
with time interval $\timeInterval=[t_0,\infty)$ or $\timeInterval=[t_0,\tend]$ with $t_0\in\RR_{\ge 0}$ and $\tend\in\RR_{>t_0}$, mass matrix $E\colon \RR_{\ge 0}\times\RR^{\fulldim}\to\RR^{\fulldim,\fulldim}$, right-hand side $\F\colon \RR_{\ge 0}\times\RR^{\fulldim}\to\RR^{\fulldim}$, and initial value $\sol_0\in\RR^{\fulldim}$.
In the following we consider mainly the case where $E$ is pointwise invertible, while we refer to the \gls{dae} literature for the more general case, see for instance \cite{KunM06}.

We call $\sol\in C(\timeInterval,\RR^{\fulldim})$ a \emph{solution of} \eqref{eq:finiteDimensionalODE} if $\sol$ is differentiable in $\timeInterval$ and satisfies \eqref{eq:finiteDimensionalODE}.
If in addition \eqref{eq:finiteDimensionalODE_IC} holds, then we call $\sol$ a \emph{solution of the initial value problem} \eqref{eq:finiteDimensionalODE_IVP}.
Based on certain assumptions on $E$ and $F$, the following theorem provides a statement about the existence and uniqueness of solutions as well as on the maximal existence interval.
The result may be shown by exploiting the pointwise invertibility of $E$ and standard \gls{ode} theory, cf.~\cite[sec.~2.4]{Vid93}.
Moreover, we emphasize that the right endpoint $\deltaMax$ of the maximal existence interval depends in general on $\sol_0$ and $t_0$.

\begin{theorem}[Maximal existence interval]
	\label{thm:globalExistenceAndUniqueness}
	Consider the system \eqref{eq:finiteDimensionalODE} with $\F\in C^1(\RR_{\ge 0}\times\RR^{\fulldim},\RR^{\fulldim})$ and pointwise invertible $E\in C^1(\RR_{\ge 0}\times\RR^{\fulldim},\RR^{\fulldim,\fulldim})$.
	Then, for each $(t_0,\sol_0)\in\RR_{\ge 0}\times\RR^{\fulldim}$ exactly one of the following two assertions is true.
	\begin{enumerate}[(i)]
		\item \label{itm:finiteBlowup}There exists a unique $\deltaMax\in\RR_{>t_0}$ such that for any $\tend\in(t_0,\deltaMax)$ the initial value problem \eqref{eq:finiteDimensionalODE_IVP} with $\timeInterval=[t_0,\tend]$ and $\sol(t_0)=\sol_0$ has a unique solution, but not for any larger end time $\tend\ge \deltaMax$.
			Furthermore, we have $\lim\limits_{t\nearrow\deltaMax}\norm{\sol(t)} = \infty$.
		\item  For any $\tend\in\RR_{>t_0}$, the initial value problem \eqref{eq:finiteDimensionalODE_IVP} with $\timeInterval=[t_0,\tend]$ and $\sol(t_0)=\sol_0$ has a unique solution.
	\end{enumerate}	 
\end{theorem}

In this work we are especially interested in the stability properties of a dynamical system.
To this end, we introduce in the following the notion of equilibrium points in general and of uniformly stable equilibrium points in particular.

\begin{definition}[Equilibrium point] 
	\label{def:equilibrium}
	For given $\F\colon \RR_{\ge 0}\times\RR^{\fulldim}\to\RR^{\fulldim}$ and pointwise invertible $E\colon \RR_{\ge 0}\times\RR^{\fulldim}\to\RR^{\fulldim,\fulldim}$, we call $\sol^*\in\RR^n$ an \emph{equilibrium point} of \eqref{eq:finiteDimensionalODE} if $\F(t,\sol^*)=0$ holds for all $t\in\RR_{\ge 0}$.
\end{definition}

\begin{definition}[Uniform stability]
	\label{def:stabilityNotions}
	Consider \eqref{eq:finiteDimensionalODE} with $\F\in C(\RR_{\ge 0}\times\RR^{\fulldim},\RR^{\fulldim})$, pointwise invertible $E\in C(\RR_{\ge 0}\times\RR^{\fulldim},\RR^{\fulldim,\fulldim})$, and equilibrium point $0\in\RR^{\fulldim}$.
	Besides, for any $(t_0,\sol_0)\in\RR_{\ge 0}\times\RR^{\fulldim}$, let the initial value problem \eqref{eq:finiteDimensionalODE_IVP} be uniquely solvable on $[t_0,\infty)$.
	We denote the evaluation of this solution at $t\in\RR_{\ge t_0}$ by $s(t,t_0,\sol_0)$.
	Then, we call the equilibrium point $0$ \emph{uniformly stable}, if for each $\epsilon \in\RR_{>0}$ there exists a $\delta\in\RR_{>0}$ such that
	\begin{equation}
		\label{eq:stabilityCriterion}
		\norm{s(t,t_0,\sol_0)}< \epsilon
	\end{equation}
	holds for all $t_0\in\RR_{\ge 0}$, $t\in\RR_{\ge t_0}$, and $\sol_0\in\RR^{\fulldim}$ with $\norm{\sol_0}<\delta$.
\end{definition}

In \cref{sec:pHMOR} the main tool used for investigating the uniform stability as in \Cref{def:stabilityNotions} is given by globally quadratic Lyapunov functions as introduced in the following definition.
We emphasize that this definition is inspired by standard Lyapunov theory for \gls{ode} systems with $E=I_{\fulldim}$, see for instance \cite[Thm.~4.10]{Kha02}, and by the \gls{ph} formulation introduced in \cite{MehM19}.

\begin{definition}[Globally quadratic Lyapunov function]
	\label{def:Lyapunov}
	We consider the system \eqref{eq:finiteDimensionalODE} with $E$ and $F$ as in \Cref{def:equilibrium}.
	The mapping $V\colon \RR_{\ge 0}\times \RR^{\fulldim}\rightarrow \RR$ is called a \emph{globally quadratic Lyapunov function} of \eqref{eq:finiteDimensionalODE} if the following conditions are satisfied.
	\begin{enumerate}[(i)]
		\item \label{itm:pHLyapunov}The function $V$ is continuously differentiable.
			Moreover, there exist a function $\zph\colon \RR_{\ge 0}\times \RR^n\rightarrow \RR^n$ and a constant $c_1\in\RR_{\ge0}$ such that for all $(t,\sol)\in\RR_{\ge 0}\times \RR^{\fulldim}$ we have
			\begin{align}
				\nabla_\sol V(t,\sol) 															&= E(t,\sol)^\top\zph(t,\sol)\\
				\label{eq:decreasingLyapunov}
				\text{and}\quad\partial_t V(t,\sol)+\zph(t,\sol)^\top \F(t,\sol) 	&\leq -c_1\norm{\sol}^2.
			\end{align}
		\item \label{itm:normEquivalence} There exist constants $c_2,c_3\in\RR_{>0}$ with
			\begin{equation*}
				c_2\lVert \sol\rVert^2 \le V(t,\sol) \le  c_3\lVert \sol\rVert^2\quad \text{for all }(t,\sol)\in\RR_{\ge 0}\times\RR^n.
			\end{equation*}
	\end{enumerate}
\end{definition}

Similarly as in the standard case $E=I_{\fulldim}$, the following theorem provides a relation between the existence of a Lyapunov function as defined in \Cref{def:Lyapunov} and uniform stability of the equilibrium point $0$.

\begin{theorem}[Lyapunov's theorem for \eqref{eq:finiteDimensionalODE}]
	\label{thm:Lyapunov}
	Consider the system \eqref{eq:finiteDimensionalODE} with $E$ and $F$ as in \Cref{def:stabilityNotions}.
	If there exists a globally quadratic Lyapunov function of \eqref{eq:finiteDimensionalODE}, then the equilibrium point $0$ is uniformly stable.
\end{theorem}

\begin{proof}
	Let $t_0\in\RR_{\ge 0}$ and $\sol_0\in\RR^{\fulldim}$ be arbitrary and let $s(t,t_0,\sol_0)$ denote the evaluation of the solution of the corresponding initial value problem \eqref{eq:finiteDimensionalODE_IVP} at $t\in\RR_{\ge t_0}$.
	Since $F$ is continuous and $E$ is continuous and pointwise invertible, we infer from \eqref{eq:finiteDimensionalODE} that the solution $s(t,t_0,\sol_0)$ is continuously differentiable with respect to $t$.
	Using this solution as well as a globally quadratic Lyapunov function $V$ of \eqref{eq:finiteDimensionalODE}, we introduce the mapping $\tilde{V}\colon \RR_{\ge t_0}\to \RR$ via $\tilde{V}(t)\vcentcolon= V(t,s(t,t_0,\sol_0))$.
	Condition \eqref{itm:pHLyapunov} in \Cref{def:Lyapunov} implies that $\tilde{V}$ is continuously differentiable and that its derivative satisfies
	\begin{align*}
		\dot{\tilde{V}}(t)	&= \partial_t V(t,s(t,t_0,\sol_0))+\nabla_\sol V(t,s(t,t_0,\sol_0))^\top\partial_ts(t,t_0,\sol_0)\\ 
								&= \partial_t V(t,s(t,t_0,\sol_0))+\zph(t,s(t,t_0,\sol_0))^\top E(t,s(t,t_0,\sol_0))\partial_ts(t,t_0,\sol_0)\\
								&= \partial_t V(t,s(t,t_0,\sol_0))+\zph(t,s(t,t_0,\sol_0))^\top \F(t,s(t,t_0,\sol_0))\\
								&\leq 0
	\end{align*}
	for all $t\in\RR_{\ge t_0}$.
	Thus, we infer
	\begin{equation*}
		V(t,s(t,t_0,\sol_0)) = \tilde{V}(t) \le \tilde{V}(t_0) = V(t_0,s(t_0,t_0,\sol_0)) = V(t_0,\sol_0)
	\end{equation*}
	for all $t\in\RR_{\ge t_0}$.
	Using condition (ii) from \Cref{def:Lyapunov}, we further obtain that there exist constants $c_2,c_3\in\RR_{>0}$ such that for all $t\in\RR_{\ge t_0}$ we have
	\begin{equation*}
 		\norm{s(t,t_0,\sol_0)}^2 \le \frac1{c_2} V(t,s(t,t_0,\sol_0)) \le \frac1{c_2} V(t_0,\sol_0) \le \frac{c_3}{c_2} \norm{\sol_0}^2.
	\end{equation*}
	This implies that \eqref{eq:stabilityCriterion} holds for instance when choosing $\delta = \sqrt{\frac{c_2}{c_3}}\epsilon$.
\end{proof}

\subsection{Port-Hamiltonian Systems}\label{sec:pH}

In the following, we only focus on finite-dimensional \gls{ph} systems without a feedthrough term.
For an overview of infinite-dimensional \gls{ph} systems, we refer to the recent survey article \cite{RasCSS20}.
Port-Hamiltonian formulations including a feedthrough term are for instance presented in \cite{MehM19,MehU23,Sch17a}.

We start by considering linear time-invariant \gls{ph} systems of the form
\begin{subequations}
	\label{eq:pHEQ}
	\begin{align}
		\label{eq:pHEQ_stateEq}
		E\dot{\sol}(t) 	&= \left((J-R)Q-K\right)\sol(t)+Bu(t),\\
		y(t)				&= B^\top Q\sol(t)
	\end{align}
\end{subequations}
for all $t\in\timeInterval$, with $B\in\RR^{\fulldim,m}$ and $E,J,R,Q,K\in\RR^{\fulldim,\fulldim}$ satisfying
\begin{equation}
	\label{eq:pHEQProperties}
	J = -J^\top,\quad R = R^\top \geq 0,\quad E^\top Q=Q^\top E\geq 0,\quad Q^\top K = -K^\top Q.
\end{equation}
We note that the structure \eqref{eq:pHEQ}--\eqref{eq:pHEQProperties} is a linear special case of \eqref{eq:nonlinearPHDAE_Q}--\eqref{eq:nonlinearPHDAEProperties_Q_new} with time-invariant coefficients and quadratic Hamiltonian $\ham\colon \RR^\fulldim\to \RR$ defined via $\ham(\sol)\vcentcolon=\frac{1}{2}\sol^\top E^\top Q\sol$.
Moreover, we emphasize that the matrix $K$ may be removed in the special case of an invertible $Q$ matrix, via replacing $J$ by $\tilde{J} \vcentcolon= J-KQ^{-1}$.

The properties \eqref{eq:pHEQProperties} imply that the Hamiltonian $\ham$ is a non-negative function, which in particular may only increase along solutions of \eqref{eq:pHEQ_stateEq} if the input $u$ and the output $y$ do not vanish.
To see this, let $u$ be such that \eqref{eq:pHEQ_stateEq} has a solution $\sol$ in $C^1(\timeInterval,\RR^{\fulldim})$.
Then, exploiting \eqref{eq:pHEQProperties} we obtain the so-called \emph{dissipation inequality}
\begin{equation}
	\label{eq:dissipationIneq_LTI}
	\begin{aligned}
		\frac{\diff}{\diff t}(\ham\circ \sol)(t) 	= -\sol(t)^\top Q^\top RQ\sol(t)+y(t)^\top u(t) \leq y^\top(t)u(t)
	 \end{aligned}
\end{equation}
for all $t\in\timeInterval$.
Usually, the Hamiltonian represents the stored energy of the system and \eqref{eq:dissipationIneq_LTI} corresponds to a power balance, where the term $\sol^\top Q^\top RQ\sol$ describes the internal energy \emph{dissipation} and the \emph{supply rate} $y^\top u$ the energy exchange with the environment or with other subsystems, see for instance \cite{SchJ14}.
Furthermore, systems for which a dissipation inequality of the form \eqref{eq:dissipationIneq_LTI} holds are typically called \emph{passive}, see for instance \cite{ByrIW91}.

The class of nonlinear \gls{ph} systems introduced in \cref{sec:problemSetting} is a special case of the class of \gls{ph} systems of the form
\begin{subequations}
	\label{eq:nonlinearPHDAE}
	\begin{align}
		\label{eq:nonlinearPHDAE_stateEq}
		E(t,\sol(t))\dot{\sol}(t)+\rph(t,\sol(t)) 	&= (J(t,\sol(t))-R(t,\sol(t)))\zph(t,\sol(t))+B(t,\sol(t))u(t),\\
		y(t) 												&= B(t,\sol(t))^\top\zph(t,\sol(t))
	\end{align}
\end{subequations}
for all $t\in\timeInterval$, with $E,J,R\in C(\RR_{\ge 0}\times \RR^n, \RR^{n,n})$, $\rph,\zph\in C(\RR_{\ge 0}\times \RR^n,\RR^n)$, and $B\in C(\RR_{\ge 0}\times \RR^n, \RR^{n,m})$.
Here, we require that the associated Hamiltonian $\ham\in C^1(\RR_{\ge 0}\times \RR^n)$ and $E$, $J$, $R$, $\rph$, $\zph$ satisfy pointwise
\begin{equation}
	\label{eq:nonlinearPHDAEProperties_new}
	\begin{aligned}
		J=-J^\top,\quad R=R^\top \ge 0, \quad \nabla_{\sol}\ham = E^\top \zph, \quad \partial_t \ham = \rph^\top\zph,
	\end{aligned}
\end{equation}
cf.~\cite{MehM19}.
In particular, we arrive at the special case presented in \cref{sec:problemSetting} if there exists $Q\in C(\RR_{\ge 0}\times \RR^{\fulldim},\RR^{\fulldim,\fulldim})$ satisfying $\zph(t,\sol)=Q(t,\sol)\sol$ for all $(t,\sol)\in\RR_{\ge 0}\times\RR^{\fulldim}$.
As demonstrated in \cref{sec:pHMOR}, this special case is of particular interest in the context of \gls{mor}, as it allows to derive a structure-preserving Petrov--Galerkin projection framework.

An important property of the nonlinear \gls{ph} system \eqref{eq:nonlinearPHDAE}--\eqref{eq:nonlinearPHDAEProperties_new} is that the corresponding Hamiltonian satisfies a dissipation inequality.
More precisely, for a given solution $\sol\in C^1(\timeInterval,\RR^{\fulldim})$ of \eqref{eq:nonlinearPHDAE_stateEq}, the function $\hamAux\colon \timeInterval\to\RR$ defined via $\hamAux(t) \vcentcolon= \ham(t,\sol(t))$ satisfies
\begin{equation}
	\label{eq:nonlinearPHDAE_dissipationIneq}
	\begin{aligned}
		\frac{\diff \hamAux}{\diff t}(t) = -\zph(t,\sol(t))^\top R(t,\sol(t)) \zph(t,\sol(t))+y(t)^\top u(t) \le y(t)^\top u(t)
	\end{aligned}
\end{equation}
for all $t\in\timeInterval$, cf.~\cite{MehM19}.
This dissipation inequality is an important property for the investigation of stability as well as the existence and uniqueness of solutions of the state equation \eqref{eq:nonlinearPHDAE_stateEq} with $u=0$ and pointwise invertible $E$.

\begin{theorem}[Stability of \eqref{eq:nonlinearPHDAE_stateEq}]
	\label{thm:globalExistenceAndUniqueness_pH}
	Consider the system \eqref{eq:nonlinearPHDAE_stateEq} with vanishing input $u=0$, $J,R\in C^1(\RR_{\ge 0}\times \RR^{\fulldim}, \RR^{\fulldim,\fulldim})$, $\rph,\zph\in C^1(\RR_{\ge 0}\times \RR^{\fulldim},\RR^{\fulldim})$, and pointwise invertible $E\in C^1(\RR_{\ge 0}\times \RR^{\fulldim}, \RR^{\fulldim,\fulldim})$.
	Furthermore, let \eqref{eq:nonlinearPHDAEProperties_new} be satisfied pointwise for some Hamiltonian $\ham\in C^1(\RR_{\ge 0}\times \RR^{\fulldim})$, which additionally fulfills condition \eqref{itm:normEquivalence} in \Cref{def:Lyapunov} with $\ham = V$.
	Then, the following assertions hold.
	\begin{enumerate}[(i)]
		\item \label{itm:pH_existenceAndUniqueness}For each initial value $\sol_0\in\RR^{\fulldim}$ and for any time interval $\timeInterval = [t_0,\tend]$ with $t_0\in\RR_{\ge 0}$ and $\tend\in\RR_{>t_0}$, the initial value problem associated with \eqref{eq:nonlinearPHDAE_stateEq}, $u=0$, and $\sol(t_0)=\sol_0$ has a unique solution on $\timeInterval$.
		\item \label{itm:pH_stability}If $\rph(t,0)=(J(t,0)-R(t,0))\zph(t,0)$ holds for all $t\in\RR_{\ge 0}$, then \eqref{eq:nonlinearPHDAE_stateEq} with $u=0$ has a uniformly stable equilibrium point at the origin.
	\end{enumerate}
\end{theorem}

\begin{proof}
	\begin{enumerate}[(i)]
		\item Using \Cref{thm:globalExistenceAndUniqueness} we conclude that, for a given initial value $\sol_0\in\RR^n$ and initial time $t_0\in\RR_{\ge 0}$, 
			the corresponding initial value problem associated with \eqref{eq:nonlinearPHDAE_stateEq} and $u=0$ is either uniquely solvable on any time interval $\timeInterval=[t_0,\tend]$ with $\tend\in\RR_{>t_0}$ or there is a maximal existence interval $[t_0,\deltaMax)$ with $\deltaMax\in\RR_{>t_0}$ and
			\begin{equation}
				\label{eq:stateBlowup}
				\lim\limits_{t\nearrow\deltaMax}\norm{\sol(t)}=\infty.
			\end{equation}
			Let us assume that for some $(t_0,\sol_0)\in\RR_{\ge 0}\times\RR^{\fulldim}$ the latter statement is true.
			Then, since the Hamiltonian satisfies condition \eqref{itm:normEquivalence} in \Cref{def:Lyapunov}, \eqref{eq:stateBlowup} implies
			\begin{equation*}
				\lim\limits_{t\nearrow\deltaMax}\ham(t,\sol(t)) =\infty.
			\end{equation*}
			However, this contradicts the inequality $\ham(t,\sol(t))\le \ham(t_0,\sol_0)$, which holds for any $t\ge t_0$ and follows from the dissipation inequality \eqref{eq:nonlinearPHDAE_dissipationIneq} in the case $u=0$.
			Thus, assertion \eqref{itm:pH_existenceAndUniqueness} holds.
		\item First, we note that the equation $\rph(\cdot,0)=(J(\cdot,0)-R(\cdot,0))\zph(\cdot,0)$ implies that $0\in\RR^{\fulldim}$ is an equilibrium point of \eqref{eq:nonlinearPHDAE_stateEq} with $u=0$.
			Furthermore, using \eqref{eq:nonlinearPHDAEProperties_new} we infer that the Hamiltonian satisfies not only condition \eqref{itm:normEquivalence} in \Cref{def:Lyapunov}, but also condition~\eqref{itm:pHLyapunov} with $c_1=0$.
			Thus, the Hamiltonian is a globally quadratic Lyapunov function of \eqref{eq:nonlinearPHDAE_stateEq} with $u=0$ and, hence, the claim follows by applying \Cref{thm:Lyapunov}.
	\end{enumerate}
\end{proof}

A notable special case of the nonlinear \gls{ph} structure \eqref{eq:nonlinearPHDAE}--\eqref{eq:nonlinearPHDAEProperties_new} and a generalization of the linear time-invariant structure \eqref{eq:pHEQ}--\eqref{eq:pHEQProperties} is given by linear time-varying systems of the form
\begin{subequations}
	\label{eq:pHLTV}
	\begin{align}
		\label{eq:pHLTV_stateEq}
		E(t)\dot{\sol}(t) 	&= \left((J(t)-R(t))Q(t)-K(t)\right)\sol(t)+B(t)u(t),\\
		y(t)				&= B(t)^\top Q(t)\sol(t)
	\end{align}
\end{subequations}
for all $t\in\timeInterval$, with $E,Q,J,R,K\in C(\RR_{\ge 0},\RR^{n,n})$ and $B\in C(\RR_{\ge 0},\RR^{n,m})$ satisfying $E^\top Q\in C^1(\RR_{\ge 0},\RR^{\fulldim,\fulldim})$  and pointwise
\begin{equation}
	\label{eq:pHLTVconditions}
	\begin{aligned}
		J = -J^\top,\quad R = R^\top \geq 0,\quad E^\top Q=Q^\top E\geq 0,\quad \frac{\diff }{\diff t}(Q^\top E) = Q^\top K+K^\top Q,
	\end{aligned}
\end{equation}
see also \cite{BeaMXZ18} for an alternative formulation of linear time-varying \gls{ph} systems.
Especially, we note that $0$ is an equilibrium point of \eqref{eq:pHLTV_stateEq} with $u=0$ and, thus, the additional requirement $\rph(\cdot,0)=(J(\cdot,0)-R(\cdot,0))\zph(\cdot,0)$ in \Cref{thm:globalExistenceAndUniqueness_pH}\eqref{itm:pH_stability} may me omitted in the linear case.
Furthermore, the assumption that the Hamiltonian $\ham\in C^1(\RR_{\ge 0}\times \RR^{\fulldim})$ defined via $\ham(t,\sol)\vcentcolon=\frac12\sol^\top E(t)^\top Q(t)\sol$ satisfies condition~\eqref{itm:normEquivalence} in \Cref{def:Lyapunov} is equivalent to the existence of $\tilde{c}_1,\tilde{c}_2\in\RR_{>0}$ with
\begin{equation*}
	\maxSingularVal(E(t)^\top Q(t)) \le \tilde{c}_1\quad \text{and}\quad \minSingularVal(E(t)^\top Q(t)) \ge \tilde{c}_2\quad \text{for all }t\in\RR_{\ge 0}.
\end{equation*}

\subsection{Projection-Based Model Reduction}
\label{sec:projectionMOR}

Classical \gls{mor} methods typically involve a projection of the \gls{fom} onto a low-dimensional linear subspace.
In the following, we consider a \gls{fom} of the form
\begin{equation}
	\label{eq:FOM_preliminaries}
	\dot{\sol}(t) = \F\left(t,\sol(t)\right)\text{ for all }t\in\timeInterval,\quad \sol(t_0) = \sol_0,
\end{equation}
with $\F\colon \RR_{\ge 0}\times\RR^{\fulldim}\rightarrow\RR^{\fulldim}$, $\sol_0\in\RR^{\fulldim}$, and $\sol\colon \timeInterval\to \RR^{\fulldim}$.
For instance, \eqref{eq:nonlinearPHDAE_Q_stateEq} may be written in this form with
\begin{equation*}
	\F(t,\sol) \vcentcolon= E(t,\sol)^{-1}\left((J(t,\sol)-R(t,\sol))Q(t,\sol)\sol+B(t,\sol)u(t)-\rph(t,\sol)\right),
\end{equation*}
provided that $E$ is pointwise invertible.
For the following considerations, we assume that we have given a suitable $\rDim$-dimensional linear subspace $\calV\subseteq \RR^{\fulldim}$ which is parametrized by a matrix $\Vr\in\RR^{\fulldim,\rDim}$ via $\calV = \colsp(\Vr)$.
A common approach for deriving a \gls{rom} is a Galerkin projection.
For this purpose, we substitute the linear approximation ansatz
\begin{equation}
	\label{eq:linearAnsatz}
	\sol(t) \approx \Vr\redState(t)
\end{equation}
into the \gls{fom} \eqref{eq:FOM_preliminaries} and obtain the residual
\begin{equation}
	\label{eq:residualGalerkin}
	\Vr\dot{\redState}(t)-\F\left(t,\Vr\redState(t)\right)
\end{equation}
at $t\in\timeInterval$.
An evolution equation for $\redState$ is then obtained by enforcing the residual to be orthogonal to $\calV = \colsp(\Vr)$.
In addition, the initial value of $\redState$ may be derived via an orthogonal projection of $\sol_0$ onto $\calV$.
The resulting \gls{rom} reads
\begin{equation}
	\label{eq:GROM}
	\MPOD\dot{\redState}(t) = \FRed(t,\redState(t)),\quad 
	\redState(t_0) = \redState_0 \quad \text{for all } t\in\timeInterval
\end{equation}
with \gls{rom} state $\redState\colon \timeInterval\to\RR^{\rDim}$, mass matrix $\MPOD\in\RR^{\rDim,\rDim}$, right-hand side $\FRed\colon\RR_{\ge 0}\times\RR^{\rDim}\to\RR^{\rDim}$, and initial value $\redState_0\in\RR^{\rDim}$ defined via
\begin{equation}
	\label{eq:GROMdefinitions}
	\MPOD \vcentcolon= \Vr^\top \Vr,\quad \FRed(t,\redState) \vcentcolon= \Vr^\top\F\left(t,\Vr\redState\right),\quad \MPOD\redState_0 \vcentcolon= \Vr^\top \sol_0.
\end{equation}
Usually, $\Vr$ is chosen such that its columns form an orthonormal basis, which leads to $\MPOD = I_{\rDim}$.

An important property is that, for given $t\in\timeInterval$ and given $\redState(t)\in\RR^{\rDim}$, the corresponding time derivative $\dot{\redState}(t)$ determined by the Galerkin \gls{rom} \eqref{eq:GROM} is optimal in the sense that it minimizes the norm of the residual \eqref{eq:residualGalerkin}.
Since the continuous-time residual is minimized, this property is called \emph{continuous optimality} in \cite{CarBA17} to distinguish it from an alternative approach which minimizes the residual after time discretization.

We note that the Galerkin method is a special case of a Petrov--Galerkin scheme.
In general, the Petrov--Galerkin method is based on enforcing the residual to be orthogonal to another $\rDim$-dimensional subspace $\calW$ which is parametrized by a matrix $\Wr\in\RR^{\fulldim,\rDim}$ via $\calW = \colsp(\Wr)$.
Here, $\calV$ and $\calW$ are chosen such that the compatibility condition 
\begin{equation}
	\label{eq:compatibiliyPetrovGalerkin}
	\calW^\perp \cap \calV = \zset
\end{equation}
is met.
This condition is in particular satisfied in the special case $\Vr = \Wr$ of the Galerkin method.
In general, a Petrov--Galerkin projection yields a \gls{rom} of the form \eqref{eq:GROM} where the mass matrix, right-hand side, and initial value are given by
\begin{align*}
	\MPOD \vcentcolon= \Wr^\top \Vr,\quad \FRed(t,\redState) \vcentcolon= \Wr^\top\F\left(t,\Vr\redState\right),\quad \MPOD\redState_0 \vcentcolon= \Wr^\top \sol_0.
\end{align*}
In particular, the invertibility of the mass matrix $\MPOD$ is guaranteed due to the compatibility condition \eqref{eq:compatibiliyPetrovGalerkin}.

Due to the shortcomings of linear projection methods, e.g., in the context of transport-dominated systems, nonlinear projection methods have received increasing attention in recent years.
In contrast to the linear ansatz \eqref{eq:linearAnsatz}, these methods are based on general ansatzes of the form
\begin{equation}
	\label{eq:nonlinearAnsatz}
	\sol(t) \approx \gr(\redState(t)),
\end{equation}
with $\gr\in C^1(\RR^{\rDim},\RR^{\fulldim})$.
A method for constructing a \gls{rom} based on such a general ansatz has, for example, been proposed in \cite{LeeC19} and is based on residual minimization.
It is inspired by the optimality property of the Galerkin \gls{rom} mentioned after \eqref{eq:GROMdefinitions} and leads to a \gls{rom} of the form
\begin{equation*}
	\MPOD(\redState(t))\dot{\redState}(t) = \FRed(t,\redState(t)),\quad 
	\redState(t_0) = \redState_0 \quad \text{for all } t\in\timeInterval
\end{equation*}
with mass matrix $\MPOD\colon \RR^{\rDim}\to\RR^{\rDim,\rDim}$ and right-hand side $\FRed\colon\RR_{\ge 0}\times\RR^{\rDim}\to\RR^{\rDim}$ defined via
\begin{align*}
	\MPOD(\redState) \vcentcolon= (\gr'(\redState))^\top \gr'(\redState),\quad \FRed(t,\redState) \vcentcolon= (\gr'(\redState))^\top\F\left(t,\gr(\redState)\right).
\end{align*}
Here, $\gr'\colon \RR^{\rDim}\to\RR^{\fulldim,\rDim}$ denotes the derivative of $\gr$.
The choice of the initial value is more involved than in the linear case, since it is in general not clear if there exists an optimal $\redState_0\in\RR^{\rDim}$ which solves $\min_{\redState\in\RR^{\rDim}}\norm{\gr(\redState)-\sol_0}$.
For instance, there exists no minimizer in the special case 
\begin{equation*}
	\fulldim=2,\quad \rDim=1,\quad \sol_0 = 
	\begin{bmatrix}
		0\\
		0
	\end{bmatrix}
	,\quad g(\redState) \vcentcolon= e^{\redState}
	\begin{bmatrix}
		1\\
		1
	\end{bmatrix}
	.
\end{equation*}
To avoid such problems, in \cite[Rem.~3.1]{LeeC19} the authors propose to add a suitable shift to the ansatz \eqref{eq:nonlinearAnsatz}, which ensures that the \gls{fom} initial value is approximated without any error.
This is possible for any choice of $\redState_0$.

So far, we have only discussed the construction of the \gls{rom}, once suitable subspaces or manifolds have been determined.
On the other hand, the determination of these subspaces or manifolds is not in the main focus of this paper.
For identifying suitable linear subspaces, there are numerous techniques provided in the \gls{mor} literature, for instance, balanced truncation \cite{BenB17,Moo81}, transfer function interpolation \cite{AntBG20,BenF21}, proper orthogonal decomposition \cite{BerHL93,GubV17}, and reduced basis methods \cite{HesRS16,QuaMN16}.
Approaches for determining suitable nonlinear manifolds are, for example, proposed in \cite{BarF22,BlaSU20,LeeC19,Gur08}.

Even though the \glspl{rom} obtained via projection usually have much fewer equations and unknowns than the corresponding \gls{fom}, the evaluation of the \gls{rom} often still scales with the dimension of the \gls{fom}.
The reason for this is that the definitions of the \gls{rom} coefficient matrices and right-hand sides formally involve the corresponding \gls{fom} quantities.
If a linear approximation ansatz is used and the \gls{fom} is linear, then this issue may be usually circumvented by precomputing the \gls{rom} coefficient matrices.
A similar approach is also possible for certain classes of \gls{fom} nonlinearities, see for instance \cite{KraW19} and the references therein.
Besides, so-called hyperreduction methods may be used to further approximate the \gls{rom} in order to render its evaluation fast, cf.~\cite{AllK17,AstWWB08,BarMNP04,CarBF11,ChaS10,FarACC14}.
While structure-preserving hyperreduction methods for \gls{ph} systems are not within the scope of this paper, we refer to \cite{ChaBG16} for a structure-preserving variant of the discrete empirical interpolation method proposed in \cite{ChaS10}.

The literature on structure-preserving model reduction for \gls{ph} systems has so far mainly focused on linear time-invariant \glspl{fom} of the form \eqref{eq:pHEQ}--\eqref{eq:pHEQProperties} with $K=0$, see for instance \cite{MehU23} and the references therein.
In this context, many projection-based schemes are based on a Petrov--Galerkin projection with $\Wr = Q\Vr$.
This is also possible for the case $K\ne 0$ and results in a \gls{rom} of the form
\begin{subequations}
	\label{eq:pHEQROM_LTIMOR}
	\begin{align}
		\label{eq:pHEQROM_LTIMOR_stateEq}
		\redE\dot{\redState}(t) 	&= (\redJ-\redR)\redQ\redState(t)+\redB u(t),\\
		\redy(t)							&= \redB^\top \redQ\redState(t)
	\end{align}
\end{subequations}
with 
\begin{equation}
	\label{eq:pHEQROM_LTIMOR_coeffs}
	\redE \vcentcolon= \Vr^\top Q^\top E\Vr,\quad \redJ \vcentcolon= \Vr^\top Q^\top \left(JQ-K\right)\Vr,\quad \redR \vcentcolon= \Vr^\top Q^\top RQ\Vr,\quad \redQ \vcentcolon= I_{\rDim},\quad \redB \vcentcolon= \Vr^\top Q^\top B.
\end{equation}
Here, $\Vr$ may be chosen in different ways depending on the \gls{mor} method of choice.
While it is straightforward to show that \eqref{eq:pHEQROM_LTIMOR} inherits the \gls{ph} structure of the \gls{fom}, one is often also interested in preserving the algebraic constraints in the case where $E$ is singular, cf.~\Cref{rem:algebraicConstraints}.
Furthermore, we emphasize that in the special case $Q=I_{\fulldim}$, the Petrov--Galerkin scheme reduces to a Galerkin projection.
On the other hand, in the case $E=I_{\fulldim}$, authors often also enforce $\redE$ to equal the identity matrix, for example, by using $\Wr = Q\Vr(\Vr^\top Q \Vr)^{-1}$ instead of $\Wr = Q\Vr$, see for instance \cite{GugPBS12}.
A similar Petrov--Galerkin approach has been employed in \cite{ChaBG16} to obtain a structure-preserving \gls{mor} scheme for nonlinear \gls{ph} systems.

\section{Structure-Preserving Model Reduction}
\label{sec:pHMOR}

In the following we consider structure-preserving \gls{mor} approaches for \gls{ph} systems using different approximation ansatzes.
We start by considering linear time-invariant and linear time-varying ansatzes in \cref{sec:linearPHMOR} and discuss the nonlinear class of ansatzes \eqref{eq:timeDependentFactorizableAnsatz} in \cref{sec:factorizableMOR}.
Finally, in \cref{sec:separableMOR} we consider separable approximation ansatzes, which correspond to a special case of \eqref{eq:timeDependentFactorizableAnsatz} for which it is possible to obtain \glspl{rom} which are at the same time continuously optimal and \gls{ph}.

\subsection{Linear Approximation Ansatz}\label{sec:linearPHMOR}

We start by considering the special case of \eqref{eq:timeDependentFactorizableAnsatz} where $\Vr$ is a constant matrix.
This special case corresponds to a classical linear time-invariant approximation ansatz which is typically employed in the structure-preserving \gls{mor} literature.
When considering linear time-invariant \gls{ph} systems of the form \eqref{eq:pHEQ}--\eqref{eq:pHEQProperties} with invertible $E$ and $Q$, structure-preserving \gls{mor} based on a linear time-invariant ansatz may be achieved via a Petrov--Galerkin projection, cf.~\eqref{eq:pHEQROM_LTIMOR}--\eqref{eq:pHEQROM_LTIMOR_coeffs} in \cref{sec:projectionMOR}.
In the following theorem, it is stated that the corresponding \gls{rom} obtained via the special Petrov--Galerkin projection with $\Wr = Q\Vr$ is optimal in the sense that the residual is minimized with respect to the weighted $E^{-\top} Q^\top$-norm.

\begin{theorem}[Structure-preserving \gls{mor} for \eqref{eq:pHEQ} using a linear time-invariant approximation ansatz]
	\label{thm:linearPH_linearMOR}
	Consider the \gls{ph} system \eqref{eq:pHEQ} with $E,J,R,Q,K$ satisfying \eqref{eq:pHEQProperties} and let $E$ and $Q$ be invertible.
	Furthermore, let $\Vr\in\RR^{\fulldim,\rDim}$ with $\rDim\in\NN_{\le \fulldim}$ be a matrix with full column rank and let \eqref{eq:pHEQROM_LTIMOR}--\eqref{eq:pHEQROM_LTIMOR_coeffs} be a corresponding \gls{rom}.
	Besides, we introduce $\res\colon \RR^\rDim\times \RR^\rDim\times \RR^m\to \RR^\fulldim$ via
	\begin{equation*}
		\res(\eta_1,\eta_2,\eta_3) \vcentcolon= E\Vr\eta_1-\left((J-R)Q-K\right)\Vr\eta_2-B\eta_3,
	\end{equation*}
	i.e., $\res$ is defined such that $\res(\dot{\redState}(t),\redState(t),u(t))$ coincides with the residual at $t\in\timeInterval$.
	Then, the \gls{rom} \eqref{eq:pHEQROM_LTIMOR} is optimal in the sense that any solution $\redState$ of \eqref{eq:pHEQROM_LTIMOR_stateEq} satisfies
	\begin{equation}
		\label{eq:linearPH_linearMOR_resMin}
		\dot{\redState}(t) \in \argmin_{\eta_1\in\RR^\rDim} \frac12\norm{\res(\eta_1,\redState(t),u(t))}_{E^{-\top}Q^\top}^2
	\end{equation}
	for all $t\in \timeInterval$ and for any input signal $u\colon \RR_{\ge 0}\to\RR^m$ which admits a solution of the \gls{rom} state equation \eqref{eq:pHEQROM_LTIMOR_stateEq}.
\end{theorem}

\begin{proof}
	First, we note that $E^{-\top}Q^\top = E^{-\top} (Q^\top E)E^{-1}$ is symmetric and positive definite since $Q^\top E$ is, and, thus, $\norm{\cdot}_{E^{-\top}Q^\top}$ is indeed a norm.
	Furthermore, the first-order necessary optimality condition reads
	\begin{equation*}
		\Vr^\top Q^\top E\Vr\eta_1 = \Vr^\top Q^\top \left((J-R)Q-K\right)\Vr\redState(t)+\Vr^\top Q^\top Bu(t)
	\end{equation*}
	and this condition is even sufficient since the Hessian $\Vr^\top Q^\top E\Vr$ is constant and positive definite.
	Finally, the comparison of the first-order optimality condition with \eqref{eq:pHEQROM_LTIMOR_stateEq} yields the claim.
\end{proof}

We note that the \gls{ph} structure of the \gls{rom} \eqref{eq:pHEQROM_LTIMOR} is a special case of the one of the \gls{fom} with $\redQ=I_{\fulldim}$, $\redK=0$, and symmetric and positive definite $\redE$.
Consequently, the \gls{rom} is passive, cf.~\cref{sec:pH}, and its state equation \eqref{eq:pHEQROM_LTIMOR_stateEq} with $u=0$ has a stable equilibrium point at the origin, see for instance \cite[Thm.~1]{GilMS18}.

\begin{remark}[Another motivation for using the $E^{-\top}Q^\top$-norm in \eqref{eq:linearPH_linearMOR_resMin}]
	\label{rem:motivationForWeightedResNorm}
	In \Cref{thm:linearPH_linearMOR} it is stated that minimizing the residual in the $E^{-\top}Q^\top$-norm yields a port-Hamiltonian \gls{rom}.
	An alternative motivation for using this specific weighted norm is given by residual-based error bounds.
	For this purpose, we first note that in the case where $E$ and $Q$ are invertible and $E^\top Q$ is symmetric and positive semi-definite, the Hamiltonian coincides up to the prefactor $\frac12$ with the squared $E^\top Q$-norm, i.e., $\ham(\sol) = \frac12\norm{\sol}_{E^\top Q}^2$ for all $\sol\in\RR^{\fulldim}$.
	When choosing this norm for measuring the error $\error \vcentcolon= \sol-\Vr\redState$, we obtain by similar arguments as in the proofs of \cite[Thm.~5.10]{BlaSU20} and \cite[Thm.~4.2]{NagS04} the error bound
	\begin{equation}
		\label{eq:errorBound}
		\norm{\error(t)}_{E^\top Q} \le Me^{\omega t}\left(\norm{\error(0)}_{E^\top Q}+\integral{0}{t}{e^{-\omega s}\norm{\res(\dot\redState(s),\redState(s),u(s))}_{E^{-\top}Q^\top}}s\right)
	\end{equation}
	for all $t\in \timeInterval = [0,\tend]$, where $M\in\RR_{\ge 1}$ and $\omega\in\RR$ satisfy
	\begin{equation*}
		\norm{e^{E^{-1}((J-R)Q-K)t}}_{E^\top Q}\le Me^{\omega t}\quad \text{for all }t\in \timeInterval.
	\end{equation*}
	Consequently, the error bound in \eqref{eq:errorBound} increases with increasing values of the $E^{-\top}Q^\top$-norm of the residual.
\end{remark}

We proceed by considering more general linear approximation ansatzes, where the modes may also depend on time, i.e.,
\begin{equation}
	\label{eq:LTVAnsatz}
	\sol(t) \approx \Vr(t)\redState(t),
\end{equation}
with given $\Vr\colon \RR_{\ge 0}\to \RR^{\fulldim,\rDim}$.
Such ansatzes are, for instance, relevant for transport-dominated systems where the advection speed is known a priori, see for instance \cite{GlaMM98}.

In the following, it is demonstrated that a structure-preserving and residual-minimizing \gls{mor} scheme for linear time-varying port-Hamiltonian \glspl{fom} of the form \eqref{eq:pHLTV} may be obtained in a similar way as in the linear-time invariant case addressed in \Cref{thm:linearPH_linearMOR}.
The stability of the resulting \gls{rom} is discussed in the upcoming \Cref{rem:nonlinearPHDAE_Q_factorizableMOR_stability_specialCases}.

\begin{theorem}[Structure-preserving \gls{mor} for \eqref{eq:pHLTV} using a linear time-varying approximation ansatz]
	\label{thm:LTVPH_timeMOR}
	Consider the \gls{ph} system \eqref{eq:pHLTV} with $E,K,J,R,Q$ satisfying pointwise \eqref{eq:pHLTVconditions} and let $E$ and $Q$ be pointwise invertible.
	Furthermore, let $\Vr\in C^1(\RR_{\ge 0},\RR^{\fulldim,\rDim})$ with $\rDim\in\NN_{\le \fulldim}$ have pointwise full column rank and let
	\begin{subequations}
		\label{eq:LTVPHROM_timeMOR}
		\begin{align}
			\label{eq:LTVPHROM_timeMOR_stateEq}
			\redE(t)\dot{\redState}(t) &= \left((\redJ(t)-\redR(t))\redQ-\redK(t)\right)\redState(t)+\redB(t) u(t),\\
			\redy(t)															&= \redB(t)^\top\redQ\redState(t),
		\end{align}
	\end{subequations}
	for all $t\in\timeInterval$ be a corresponding \gls{rom} with coefficients $\redE,\redK,\redJ,\redR\colon\RR_{\ge 0}\to\RR^{\rDim,\rDim}$, $\redQ\in\RR^{\rDim,\rDim}$, and $\redB\colon\RR_{\ge 0}\to\RR^{\rDim,m}$ defined as
	\begin{equation}
		\label{eq:LTVPHROM_timeMOR_ROMMatrices}
		\begin{aligned}
			\redE &\vcentcolon= \Vr^\top Q^\top E\Vr,\quad 		&\redK &\vcentcolon= \Vr^\top Q^\top\left(K\Vr+E\dot\Vr\right),\quad 	&\redJ 	&\vcentcolon= \Vr^\top Q^\top JQ\Vr,\\
			\redR &\vcentcolon= \Vr^\top Q^\top RQ\Vr,\quad 	&\redQ &\vcentcolon= I_{\rDim},\quad													&\redB 	&\vcentcolon= \Vr^\top Q^\top B.
		\end{aligned}
	\end{equation}
	Besides, we introduce the residual mapping $\res\colon \RR_{\ge 0}\times\RR^\rDim\times \RR^\rDim\times \RR^m\to \RR^\fulldim$ via
	\begin{align*}
		&\res(t,\eta_1,\eta_2,\eta_3)\\
		&\vcentcolon= E(t)\Vr(t)\eta_1+\left(E(t)\dot\Vr(t)+(K(t)-(J(t)-R(t))Q(t))\Vr(t)\right)\eta_2-B(t)\eta_3.
	\end{align*}
	Then, the following assertions hold.
	\begin{enumerate}[(i)]
		\item \label{itm:LTVPHROM_timeMOR_structurePreservation}The \gls{rom} coefficients satisfy pointwise
			\begin{equation*}
				\begin{aligned}
					\redJ = -\redJ^\top,\quad \redR = \redR^\top \geq 0,\quad \redE^\top \redQ=\redQ^\top \redE> 0,\quad \frac{\diff }{\diff t}(\redQ^\top \redE) = \redQ^\top \redK+\redK^\top \redQ,
				\end{aligned}
			\end{equation*}
			i.e., the \gls{rom} \eqref{eq:LTVPHROM_timeMOR} inherits the \gls{ph} structure from the \gls{fom} \eqref{eq:pHLTV}.
		\item The \gls{rom} \eqref{eq:LTVPHROM_timeMOR} is optimal in the sense that any solution $\redState$ of \eqref{eq:LTVPHROM_timeMOR_stateEq} satisfies
			\begin{equation}
				\label{eq:LTVPHROM_timeMOR_optimality}
				\dot{\redState}(t) \in \argmin_{\eta_1\in\RR^{\rDim}} \frac12\norm{\res(t,\eta_1,\redState(t),u(t))}_{E(t)^{-\top}Q(t)^\top}^2
			\end{equation}
			for all $t\in \timeInterval$ and for any input signal $u\colon \RR_{\ge 0}\to\RR^m$ which admits a solution of the \gls{rom} state equation \eqref{eq:LTVPHROM_timeMOR_stateEq}.
	\end{enumerate}
\end{theorem}

\begin{proof}
	\begin{enumerate}[(i)]
		\item The pointwise symmetry and definiteness properties of $\redJ$, $\redR$, and $\redQ^\top\redE$ follow from the corresponding properties of $J$, $R$, and $Q^\top E$, respectively.
			Furthermore, using the last equation in \eqref{eq:pHLTVconditions}, we obtain
			\begin{align*}
				\frac{\diff }{\diff t}(\redQ^\top \redE) &= \dot\Vr^\top Q^\top E\Vr+\Vr^\top\frac{\diff }{\diff t}(Q^\top E) \Vr+\Vr^\top Q^\top E\dot\Vr\\
				&= \dot\Vr^\top Q^\top E\Vr+\Vr^\top\left(Q^\top K+K^\top Q\right)\Vr+\Vr^\top Q^\top E\dot\Vr = \redQ^\top\redK+\redK^\top \redQ.
			\end{align*}
		\item The proof follows along the lines of the proof of \Cref{thm:linearPH_linearMOR}, where the major difference is that the first-order necessary optimality condition of the minimization problem \eqref{eq:LTVPHROM_timeMOR_optimality} for fixed $t\in\timeInterval$ is given by
			\begin{align*}
				&\Vr(t)^\top Q(t)^\top E(t)\Vr(t)\eta_1+\Vr(t)^\top Q(t)^\top \left(E(t)\dot\Vr(t)+K(t)\Vr(t)\right)\redState(t)\\
				&= \Vr(t)^\top Q(t)^\top(J(t)-R(t))Q(t)\Vr(t)\redState(t)+\Vr(t)^\top Q(t)^\top B(t)u(t).
			\end{align*}
			and the Hessian is $\Vr(t)^\top Q(t)^\top E(t)\Vr(t)$.
	\end{enumerate}
\end{proof}

\subsection{Factorizable Approximation Ansatz}
\label{sec:factorizableMOR}

As mentioned in the introduction, nonlinear \gls{mor} ansatzes have received more and more attention in recent years.
In the following, we consider structure-preserving \gls{mor} schemes based on ansatzes of the form \eqref{eq:timeDependentFactorizableAnsatz}, which correspond to approximating the \gls{fom} state by a linear combination of ansatz vectors which may depend not only on time, but also on the \gls{rom} state.
An example for such an ansatz is given by polynomial ansatzes with vanishing constant term, i.e.,
\begin{equation*}
	\sol_k(t) \approx \sum_{i_1=0}^q\sum_{i_2=0}^q\cdots\sum_{i_{\rDim}=0}^q c_{k,i_1,i_2,\ldots,i_{\rDim}}\prod_{j=1}^{\rDim}\redState_j(t)^{i_j}\quad \text{for }k=1,\ldots,\fulldim
\end{equation*}
with $q\in\NN$, $c_{k,i_1,i_2,\ldots,i_{\rDim}}\in\RR$ for $i_j=0,\ldots,q$, $j=1,\ldots,\rDim$ with $c_{k,0,0,\ldots,0}=0$ for $k=1,\ldots,\fulldim$.
For instance, quadratic ansatzes have been recently investigated in \cite{BarF22}.
Another subclass of \eqref{eq:timeDependentFactorizableAnsatz}, which is especially relevant in the context of \gls{mor} for transport-dominated systems, is addressed in \cref{sec:separableMOR}.

Unfortunately, it appears to be challenging to obtain a structure-preserving \gls{mor} scheme which is at the same time residual-minimizing when using a general approximation ansatz of the form \eqref{eq:timeDependentFactorizableAnsatz}.
Instead, we focus in the following only on the structure preservation, whereas we consider in \cref{sec:separableMOR} a special case of \eqref{eq:timeDependentFactorizableAnsatz} which allows to obtain continuously optimal port-Hamiltonian \glspl{rom}.
Similarly as in the previous subsection, the structure preservation is achieved by enforcing the residual to be orthogonal to the column space of $Q\Vr$.
In the following, this is demonstrated for nonlinear time-varying port-Hamiltonian \glspl{fom} of the form \eqref{eq:nonlinearPHDAE_Q}--\eqref{eq:nonlinearPHDAEProperties_Q_new}.

\begin{theorem}[Structure-preserving \gls{mor} for \eqref{eq:nonlinearPHDAE_Q} using a factorizable approximation ansatz]
	\label{thm:nonlinearPHDAE_Q_factorizableMOR}
	Consider the \gls{ph} system \eqref{eq:nonlinearPHDAE_Q} with $E,\rph,J,R,Q$ and the associated Hamiltonian $\ham$ satisfying \eqref{eq:nonlinearPHDAEProperties_Q_new}.
	Furthermore, let $\Vr\colon \RR_{\ge 0}\times \RR^{\rDim}\rightarrow\RR^{\fulldim,\rDim}$ with $\rDim\in\NN_{\le \fulldim}$ be continuously differentiable and consider the corresponding \gls{rom} of the form \eqref{eq:nonlinearPHDAE_Q_factorizableMOR} with coefficient functions defined via
	\begin{equation}
		\label{eq:nonlinearPHDAE_Q_factorizableMOR_ROMMatrices}
		\begin{aligned}
			\redE(t,\redState) 		&\vcentcolon= \Vr(t,\redState)^\top Q(t,\Vr(t,\redState)\redState)^\top E(t,\Vr(t,\redState)\redState)\left(\Vr(t,\redState)+\widehat{\Vr}(t,\redState)\redState\right),\\ 
			\redJ(t,\redState) 		&\vcentcolon= \Vr(t,\redState)^\top Q(t,\Vr(t,\redState)\redState)^\top J(t,\Vr(t,\redState)\redState)Q(t,\Vr(t,\redState)\redState)\Vr(t,\redState),\\
			\redR(t,\redState) 		&\vcentcolon= \Vr(t,\redState)^\top Q(t,\Vr(t,\redState)\redState)^\top R(t,\Vr(t,\redState)\redState)Q(t,\Vr(t,\redState)\redState)\Vr(t,\redState),\\
			\redrph(t,\redState) 	&\vcentcolon= \Vr(t,\redState)^\top Q(t,\Vr(t,\redState)\redState)^\top \left(\rph(t,\Vr(t,\redState)\redState)+E(t,\Vr(t,\redState)\redState)\partial_t\Vr(t,\redState)\redState\right),\\ 
			\redQ(t,\redState) &\vcentcolon= I_{\rDim},\\ 
			\redB(t,\redState) &\vcentcolon= \Vr(t,\redState)^\top Q(t,\Vr(t,\redState)\redState)^\top B(t,\Vr(t,\redState)\redState).
		\end{aligned}
	\end{equation}
	Here, $\widehat\Vr\colon \RR_{\ge 0}\times\RR^\rDim\to\Hom(\RR^\rDim,\RR^{\fulldim,\rDim})$ is defined via
	\begin{equation}
		\label{eq:widehatVr_factorizableMOR}
		\widehat\Vr(t,\eta_1)(\eta_2)\eta_3 \vcentcolon= \partial_{\redState}\Vr(t,\eta_1)(\eta_3)\eta_2
	\end{equation}
	for all $(t,\eta_1,\eta_2,\eta_3)\in\RR_{\ge 0}\times\RR^{\rDim}\times\RR^{\rDim}\times\RR^{\rDim}$.
	Moreover, we consider the \gls{rom} Hamiltonian $\redHam\colon\RR_{\ge 0}\times\RR^{\rDim}\to\RR$ as in \eqref{eq:redHam}.
	Then, $\redHam$ is continuously differentiable and the \gls{rom} coefficients satisfy 
	\begin{equation}
		\label{eq:nonlinearPHDAEROMProperties}
		\begin{aligned}
			\redJ(t,\redState) &= -\redJ(t,\redState)^\top, &\nabla_{\redState}\redHam(t,\redState) &= \redE(t,\redState)^\top \redQ(t,\redState)\redState\\
			\redR(t,\redState) &= \redR(t,\redState)^\top \ge 0,\quad &\partial_t \redHam(t,\redState) &= \redrph(t,\redState)^\top \redQ(t,\redState)\redState
		\end{aligned}
	\end{equation}
	for all $(t,\redState)\in\RR_{\ge 0}\times \RR^{\rDim}$, i.e., the \gls{rom} inherits the \gls{ph} structure from the \gls{fom} \eqref{eq:nonlinearPHDAE_Q}.
\end{theorem}

\begin{proof}
	The properties of $\redJ$ and $\redR$ follow from the corresponding properties of $J$ and $R$, respectively.
	Furthermore, $\redHam$ is continuously differentiable due to the continuous differentiability of $\ham$ and $\Vr$.
	Moreover, the relations concerning the partial derivatives of $\redHam$ follow from
	\begin{align*}
		&\partial_t \redHam(t,\redState) 	= \partial_t\ham(t,\Vr(t,\redState)\redState)+\partial_{\sol}\ham(t,\Vr(t,\redState)\redState)\partial_t\Vr(t,\redState)\redState\\
		&= \rph(t,\Vr(t,\redState)\redState)^\top Q(t,\Vr(t,\redState)\redState)\Vr(t,\redState)\redState+\redState^\top\Vr(t,\redState)^\top Q(t,\Vr(t,\redState)\redState)^\top E(t,\Vr(t,\redState)\redState)\partial_t\Vr(t,\redState)\redState\\
		&= \redrph(t,\redState)^\top\redQ(t,\redState)\redState
	\end{align*}
	and
	\begin{align*}
		\partial_{\redState} \redHam(t,\redState)\zeta &= \partial_{\sol}\ham(t,\Vr(t,\redState)\redState)\left(\Vr(t,\redState)\zeta+\partial_{\redState}\Vr(t,\redState)(\zeta)\redState\right)\\
		&= \redState^\top\Vr(t,\redState)^\top Q(t,\Vr(t,\redState)\redState)^\top E(t,\Vr(t,\redState)\redState)\left(\Vr(t,\redState)+\widehat{\Vr}(t,\redState)\redState\right)\zeta\\
		&= \left(\redE(t,\redState)^\top\redQ(t,\redState)\redState\right)^\top\zeta
	\end{align*}
	for all $(t,\redState,\zeta)\in\RR_{\ge 0}\times\RR^{\rDim}\times\RR^{\rDim}$.
\end{proof}

While the subject of \Cref{thm:nonlinearPHDAE_Q_factorizableMOR} is the \gls{ph} structure of the \gls{rom} given by \eqref{eq:nonlinearPHDAE_Q_factorizableMOR} and \eqref{eq:nonlinearPHDAE_Q_factorizableMOR_ROMMatrices}, this theorem does not address the stability of the state equation \eqref{eq:nonlinearPHDAE_Q_factorizableMOR_stateEq}.
Based on \Cref{thm:globalExistenceAndUniqueness_pH}, the following corollary provides sufficient conditions for the \gls{rom} state equation to have a uniformly stable equilibrium point at the origin.

\begin{corollary}[Stability of the \gls{rom} from \Cref{thm:nonlinearPHDAE_Q_factorizableMOR}]
	\label{cor:nonlinearPHDAE_Q_factorizableMOR_stability}
	Let the assumptions of \Cref{thm:nonlinearPHDAE_Q_factorizableMOR} be satisfied and let additionally $E$, $J$, $R$, $Q$, and $\rph$ be continuously differentiable and $\Vr$ twice continuously differentiable.
	Furthermore, let the \gls{fom} Hamiltonian $\ham$ satisfy condition~\eqref{itm:normEquivalence} in \Cref{def:Lyapunov} with $V=\ham$ and let there exist constants $\hat{c}_1,\hat{c}_2\in\RR_{>0}$ such that the singular values of $\Vr$ satisfy
	\begin{equation}
		\label{eq:boundForSingularValsOfVr_timeDependentFactorizableMOR}
		\maxSingularVal(\Vr(t,\redState)) \le \hat{c}_1\quad \text{and}\quad \minSingularVal(\Vr(t,\redState)) \ge \hat{c}_2\quad \text{for all }(t,\redState)\in\RR_{\ge 0}\times\RR^{\rDim}.
	\end{equation}
	Besides, let $E$, $Q$, and $\Vr$ be such that $\redE$ as defined in \eqref{eq:nonlinearPHDAE_Q_factorizableMOR_ROMMatrices} is pointwise invertible and let $0\in\RR^{\fulldim}$ be an equilibrium point of the \gls{fom} state equation \eqref{eq:nonlinearPHDAE_Q_stateEq} with $u=0$.
	Then, the \gls{rom} state equation \eqref{eq:nonlinearPHDAE_Q_factorizableMOR_stateEq} with $u=0$ and coefficients as in \eqref{eq:nonlinearPHDAE_Q_factorizableMOR_ROMMatrices} has a uniformly stable equilibrium point at $0\in\RR^{\rDim}$.
\end{corollary}

\begin{proof}
	First, we note that the differentiability assumptions on the \gls{fom} coefficient functions and on $\Vr$ imply that $\redE$, $\redJ$, $\redR$, and $\redrph$ are continuously differentiable.
	Furthermore, since $\ham$ satisfies condition~\eqref{itm:normEquivalence} in \Cref{def:Lyapunov} with constants $c_2,c_3\in\RR_{>0}$ and since the singular values of $\Vr$ are bounded as in \eqref{eq:boundForSingularValsOfVr_timeDependentFactorizableMOR}, we infer that also the \gls{rom} Hamiltonian $\redHam$ satisfies condition~\eqref{itm:normEquivalence} in \Cref{def:Lyapunov}, which follows from the calculation
	\begin{equation*}
		\redHam(t,\redState) = \ham(t,\Vr(t,\redState)\redState) \le c_3\norm{\Vr(t,\redState)\redState}^2 \le c_3\maxSingularVal(\Vr(t,\redState))^2\norm{\redState}^2 \le c_3\hat{c}_1^2\norm{\redState}^2
	\end{equation*}
	and
	\begin{equation*}
		\redHam(t,\redState) \ge c_2\norm{\Vr(t,\redState)\redState}^2 \ge c_2\minSingularVal(\Vr(t,\redState))^2\norm{\redState}^2 \ge c_2\hat{c}_2^2\norm{\redState}^2
	\end{equation*}
	for all $(t,\redState)\in\RR_{\ge 0}\times\RR^{\rDim}$.
	In addition, the fact that $0\in\RR^{\fulldim}$ is an equilibrium point of \eqref{eq:nonlinearPHDAE_Q_stateEq} with $u=0$ implies $\rph(t,0)=0$ for all $t\in\RR_{\ge 0}$.
	Consequently, we also have $\redrph(t,0)=0$ for all $t\in\RR_{\ge 0}$ and the claim follows from \Cref{thm:globalExistenceAndUniqueness_pH}.
\end{proof}

In general, the assumption in \Cref{cor:nonlinearPHDAE_Q_factorizableMOR_stability} that $\redE$ is pointwise invertible may be difficult to verify.
In the special case where $E^\top Q$ is pointwise symmetric and positive definite and $\Vr$ is constant with respect to its second argument and has pointwise full column rank, $\redE$ is pointwise symmetric and positive definite.
While this special case corresponds to a linear approximation ansatz as considered in the previous subsection, the situation is more involved in the nonlinear case, where $\partial_{\redState}\Vr$ does not vanish.
However, even in this case, we can at least infer that $\redE(t,0)$ is invertible for all $t\in\RR_{\ge 0}$ as long as $E^\top Q$ is pointwise symmetric and positive definite and $\Vr$ has pointwise full column rank, since the term $\widehat{\Vr}(t,\redState)\redState$ vanishes for $\redState=0$.
Thus, since $\redE$ is continuous and the set of invertible matrices is open, cf.~\cite[Prop.~1.2.1]{Far08}, in this case $\redE$ is at least invertible in a neighborhood of the equilibrium point $0$.

\begin{remark}[Special cases of \Cref{cor:nonlinearPHDAE_Q_factorizableMOR_stability}]
	\label{rem:nonlinearPHDAE_Q_factorizableMOR_stability_specialCases}
	In the special case where the \gls{fom} is a linear \gls{ph} system of the form \eqref{eq:pHEQ}--\eqref{eq:pHEQProperties} or \eqref{eq:pHLTV}--\eqref{eq:pHLTVconditions}, the assumption in \Cref{cor:nonlinearPHDAE_Q_factorizableMOR_stability} that $0$ is an equilibrium point of the \gls{fom} state equation with $u=0$ is automatically satisfied, cf.~\cref{sec:pH}.
	Furthermore, in the linear time-invariant case \eqref{eq:pHEQ}--\eqref{eq:pHEQProperties}, the assumption that the Hamiltonian $\ham$ satisfies condition~\eqref{itm:normEquivalence} in \Cref{def:Lyapunov} is equivalent to $E$ and $Q$ being invertible.
	Similarly, also in the linear time-varying case \eqref{eq:pHLTV}--\eqref{eq:pHLTVconditions}, this assumption on the Hamiltonian simplifies as mentioned at the end of \cref{sec:pH}.
	Finally, we remark that in the special case of a linear time-invariant approximation ansatz, i.e., $\Vr\in\RR^{\fulldim,\rDim}$, the condition \eqref{eq:boundForSingularValsOfVr_timeDependentFactorizableMOR} is equivalent to $\Vr$ having full column rank.
\end{remark}

\subsection{Separable Approximation Ansatz}
\label{sec:separableMOR}

A notable special case of the class of ansatzes considered in \cref{sec:factorizableMOR} is given by ansatzes of the form
\begin{equation}
	\label{eq:separableMORAnsatz}
	\sol(t) \approx \Vs(\shift(t))\amplitude(t),
\end{equation}
with given mapping $\Vs\colon \RR^{\rDim_\shift}\to \RR^{\fulldim,\rDim_\amplitude}$ and \gls{rom} state
\begin{equation}
	\label{eq:separableMOR_splitState}
	\redState = 
	\begin{bmatrix}
		\amplitude\\
		\shift
	\end{bmatrix}
\end{equation}
consisting of $\amplitude\colon \timeInterval\to \RR^{\rDim_\amplitude}$ and $\shift\colon \timeInterval\to\RR^{\rDim_\shift}$ with $\rDim\vcentcolon= \rDim_\amplitude+\rDim_\shift$.
Since the ansatz \eqref{eq:separableMORAnsatz} is linear in $\amplitude$ and possibly nonlinear in $\shift$, we call this a \emph{separable} approximation ansatz, since it is the same kind of nonlinearity as in separable nonlinear least-squares problems, see for instance \cite{GolP03} and the references therein.
Separable ansatzes occur, for instance, in some \gls{mor} approaches for transport-dominated systems, where the state is approximated by a linear combination of transformed modes and the transformations are parametrized by time-dependent path variables, here $\shift$, see for instance \cite{AndF22a,AndF22b,BlaSU20,CagMS19,RimPM23,RowM00} as well as \cref{sec:numerics}.
The fact that separable ansatzes are indeed a special case of factorizable ansatzes of the form \eqref{eq:timeDependentFactorizableAnsatz} follows from
\begin{equation}
	\label{eq:separableAnsatzIsSpecialFactorizableAnsatz}
	\Vs(\shift)\amplitude = \underbrace{
	\begin{bmatrix}
		\Vs(\shift) & 0
	\end{bmatrix}
	}_{=\vcentcolon \Vr(\redState)}\underbrace{
	\begin{bmatrix}
		\amplitude\\
		\shift
	\end{bmatrix}
	}_{= \redState}
	.
\end{equation}

As \eqref{eq:separableAnsatzIsSpecialFactorizableAnsatz} demonstrates that \eqref{eq:separableMORAnsatz} is a special case of \eqref{eq:timeDependentFactorizableAnsatz}, we could in principle use \Cref{thm:nonlinearPHDAE_Q_factorizableMOR} to obtain a structure-preserving \gls{mor} scheme.
However, the special $\Vr$ defined in \eqref{eq:separableAnsatzIsSpecialFactorizableAnsatz} has $\rDim_{\shift}$ zero columns.
Consequently, the state equation of the resulting \gls{rom} constructed as in \Cref{thm:nonlinearPHDAE_Q_factorizableMOR} has $\rDim_{\shift}$ vanishing equations, i.e., this \gls{mor} approach would always yield a singular \gls{dae} system as \gls{rom}.
In the following we demonstrate that we may alternatively obtain a port-Hamiltonian \gls{rom} via weighted residual minimization in the spirit of \cref{sec:linearPHMOR}, by exploiting the special structure of the ansatz \eqref{eq:separableMORAnsatz}.
To this end, we consider a linear time-varying \gls{fom} of the form \eqref{eq:pHLTV}--\eqref{eq:pHLTVconditions} with pointwise invertible $E$ and $Q$.
Based on the ansatz \eqref{eq:separableMORAnsatz} and a suitably weighted residual minimization approach, we obtain a \gls{rom} of the form \eqref{eq:nonlinearPHDAE_Q_factorizableMOR}, where $\redE,\redJ,\redR,\redQ,\redrph,\redB$ are defined as
\begin{subequations}
	\label{eq:pHROM_LTVFOM_separableMOR_coefficients}
	\begin{equation}
		\begin{aligned}
			\redE(t,\redState) &\vcentcolon= 
			\begin{bmatrix}
				\redE_{11}(t,\shift) 							& \redE_{12}(t,\amplitude,\shift)\\
				\redE_{12}(t,\amplitude,\shift)^\top 	& \redE_{22}(t,\amplitude,\shift)
			\end{bmatrix}
			,\\ 
			\redE_{11}(t,\shift) &\vcentcolon= \Vs(\shift)^\top Q(t)^\top E(t)\Vs(\shift)\in\RR^{\rDim_{\amplitude},\rDim_{\amplitude}},\\ 
			\redE_{12}(t,\amplitude,\shift) &\vcentcolon= \Vs(\shift)^\top Q(t)^\top E(t)\widehat\Vs(\shift)\amplitude\in\RR^{\rDim_{\amplitude},\rDim_{\shift}},\\ 
			\hspace{0.75cm}\redE_{22}(t,\amplitude,\shift) &\vcentcolon= \left(\widehat\Vs(\shift)\amplitude\right)^\top Q(t)^\top E(t)\widehat\Vs(\shift)\amplitude\in\RR^{\rDim_{\shift},\rDim_{\shift}},
		\end{aligned}
	\end{equation}
	\vspace{0.2cm}
	\begin{equation}
		\begin{aligned}
			\redJ(t,\redState) &\vcentcolon= 
			\begin{bmatrix}
				\redJ_{11}(t,\shift) 			& -\redJ_{21}(t,\amplitude,\shift)^\top\\
				\redJ_{21}(t,\amplitude,\shift)	& 0
			\end{bmatrix}
			,\\
			\redJ_{11}(t,\shift) &\vcentcolon= \Vs(\shift)^\top Q(t)^\top J(t)Q(t)\Vs(\shift)\in\RR^{\rDim_{\amplitude},\rDim_{\amplitude}},\\ 
			\hspace{1.3cm}\redJ_{21}(t,\amplitude,\shift) &\vcentcolon= \left(\widehat\Vs(\shift)\amplitude\right)^\top Q(t)^\top J(t)Q(t)\Vs(\shift)\in\RR^{\rDim_{\shift},\rDim_{\amplitude}},
		\end{aligned}
	\end{equation}
	\vspace{0.2cm}
	\begin{equation}
		\begin{aligned}
			\redR(t,\redState) &\vcentcolon= 
			\begin{bmatrix}
				\redR_{11}(t,\shift) 			& \redR_{21}(t,\amplitude,\shift)^\top\\
				\redR_{21}(t,\amplitude,\shift) & \redR_{22}(t,\amplitude,\shift)
			\end{bmatrix}
			,\\ 
			\redR_{11}(t,\shift) &\vcentcolon= \Vs(\shift)^\top Q(t)^\top R(t)Q(t)\Vs(\shift)\in\RR^{\rDim_{\amplitude},\rDim_{\amplitude}},\\ 
			\hspace{1.5cm}\redR_{21}(t,\amplitude,\shift) &\vcentcolon=\left(\widehat\Vs(\shift)\amplitude\right)^\top Q(t)^\top R(t)Q(t)\Vs(\shift)\in\RR^{\rDim_{\shift},\rDim_{\amplitude}},\\
			\redR_{22}(t,\amplitude,\shift) &\vcentcolon=\left(\widehat\Vs(\shift)\amplitude\right)^\top Q(t)^\top R(t)Q(t)\widehat\Vs(\shift)\amplitude\in\RR^{\rDim_{\shift},\rDim_{\shift}},
		\end{aligned}
	\end{equation}
	\vspace{0.2cm}
	\begin{equation}
		\begin{aligned}
			\redrph(t,\redState) &\vcentcolon= 
			\begin{bmatrix}
				\redrph_1(t,\amplitude,\shift)\\
				\redrph_2(t,\amplitude,\shift)
			\end{bmatrix}
			,\\
			\redrph_1(t,\amplitude,\shift) &\vcentcolon= \Vs(\shift)^\top Q(t)^\top K(t)\Vs(\shift)\amplitude\in\RR^{\rDim_{\amplitude}},\\ 
			\hspace{0.72cm}\redrph_2(t,\amplitude,\shift) &\vcentcolon= \left(\widehat\Vs(\shift)\amplitude\right)^\top Q(t)^\top K(t)\Vs(\shift)\amplitude\in\RR^{\rDim_{\shift}},
		\end{aligned}
	\end{equation}
	\vspace{0.2cm}
	\begin{equation}
		\redQ(t,\redState) \vcentcolon= 
		\begin{bmatrix}
			I_{\rDim_{\amplitude}} 	& 0\\
			0 									& 0
		\end{bmatrix}
		,\hspace{3.25cm}
	\end{equation}
	\vspace{0.2cm}
	\begin{equation}
		\begin{aligned}
			\redB(t,\redState) &\vcentcolon=
			\begin{bmatrix}
				\redB_1(t,\shift)\\
				\redB_2(t,\amplitude,\shift)
			\end{bmatrix}
			,\\ 
			\redB_1(t,\shift) &\vcentcolon= \Vs(\shift)^\top Q(t)^\top B(t)\in\RR^{\rDim_{\amplitude},m},\\ 
			\redB_2(t,\amplitude,\shift) &\vcentcolon= \left(\widehat\Vs(\shift)\amplitude\right)^\top Q(t)^\top B(t)\in\RR^{\rDim_{\shift},m}.\hspace{0.3cm}
		\end{aligned}
	\end{equation}
\end{subequations}
Here, we use the notation for the block components of $\redState$ as in \eqref{eq:separableMOR_splitState} and $\widehat\Vs\colon \RR^{\rDim_{\shift}}\to\Hom(\RR^{\rDim_{\amplitude}},\RR^{\fulldim,\rDim_{\shift}})$ is defined via
\begin{equation}
	\label{eq:widehatVr_separableMOR}
	\widehat\Vs(\eta_1)(\eta_2)\eta_3 \vcentcolon= \Vs'(\eta_1)(\eta_3)\eta_2\quad \text{for all }(\eta_1,\eta_2,\eta_3)\in\RR^{\rDim_{\shift}}\times\RR^{\rDim_{\amplitude}}\times \RR^{\rDim_{\shift}},
\end{equation}
where $\Vs'$ denotes the derivative of $\Vs$.
Moreover, we note that the \gls{rom} given by \eqref{eq:nonlinearPHDAE_Q_factorizableMOR} and \eqref{eq:pHROM_LTVFOM_separableMOR_coefficients} is obtained by enforcing the residual at $t\in\timeInterval$ to be orthogonal to the column space of $Q(t)[\Vs(\shift(t))\; \,\widehat\Vs(\shift(t))\amplitude(t)]$.
In the following theorem, it is stated that this leads to a \gls{rom} which is not only \gls{ph}, but also optimal in the sense of weighted residual minimization.
The boundedness of the $\amplitude$ block component of the \gls{rom} state is addressed in the upcoming \Cref{cor:nonlinearPHDAE_Q_separableMOR_stability}, see also \Cref{rem:nonlinearPHDAE_Q_factorizableMOR_stability_specialCases}.

\begin{theorem}[Structure-preserving \gls{mor} for \eqref{eq:pHLTV} using a separable approximation ansatz]
	\label{thm:LTVPH_separableMOR}
	Consider the \gls{ph} system \eqref{eq:pHLTV} with $E,K,J,R,Q$ satisfying pointwise \eqref{eq:pHLTVconditions} and let $E$ and $Q$ be pointwise invertible.
	Besides, let $\Vs\colon\RR^{\rDim_\shift}\to\RR^{\fulldim,\rDim_\amplitude}$ with $\rDim_\amplitude,\rDim_\shift\in\NN$ and $\rDim\vcentcolon=\rDim_\amplitude+\rDim_\shift\le\fulldim$ be continuously differentiable and consider the corresponding \gls{rom} \eqref{eq:nonlinearPHDAE_Q_factorizableMOR} with coefficients $\redE,\redJ,\redR,\redQ,\redrph,\redB$ as defined in \eqref{eq:pHROM_LTVFOM_separableMOR_coefficients}.
	Moreover, we define the \gls{rom} Hamiltonian $\redHam\colon \RR_{\ge 0}\times\RR^\rDim\to \RR$ via
	\begin{equation}
		\label{eq:LTVPH_separableMOR_redHam}
		\redHam(t,\redState) \vcentcolon= \frac12 \amplitude^\top \Vs(\shift)^\top E(t)^\top Q(t)\Vs(\shift)\amplitude,
	\end{equation}
	where we use the notation from \eqref{eq:separableMOR_splitState} for the block components of $\redState$.
	Furthermore, we introduce the residual mapping $\res\colon \RR_{\ge 0}\times\RR^{\rDim_{\amplitude}}\times\RR^{\rDim_{\shift}}\times\RR^{\rDim_{\amplitude}}\times\RR^{\rDim_{\shift}}\times\RR^m\to\RR^{\fulldim}$ via
	\begin{align*}
		\res(t,\eta_1,\eta_2,\eta_3,\eta_4,\eta_5) &\vcentcolon= E(t)\left(\Vs(\eta_4)\eta_1+\Vs'(\eta_4)(\eta_2)\eta_3\right)\\
		&\quad-((J(t)-R(t))Q(t)-K(t))\Vs(\eta_4)\eta_3-B(t)\eta_5.
	\end{align*}
	Then, the following assertions hold.
	\begin{enumerate}[(i)]
		\item \label{itm:LTVPH_separableMOR_structurePreservation}The \gls{rom} Hamiltonian $\redHam$ is continuously differentiable and the \gls{rom} coefficients satisfy \eqref{eq:nonlinearPHDAEROMProperties} for all $(t,\redState)\in\RR_{\ge 0}\times\RR^{\rDim}$.
		\item \label{itm:LTVPH_separableMOR_optimality}The \gls{rom} given by \eqref{eq:nonlinearPHDAE_Q_factorizableMOR} and \eqref{eq:pHROM_LTVFOM_separableMOR_coefficients} is optimal in the sense that any solution $\redState$ of \eqref{eq:nonlinearPHDAE_Q_factorizableMOR_stateEq} satisfies
			\begin{equation}
				\label{eq:pHROM_LTVFOM_separableMOR_optimality}
				\dot{\redState}(t) = 
				\begin{bmatrix}
					\dot{\amplitude}(t)\\
					\dot{\shift}(t)
				\end{bmatrix}
				\in \argmin_{
				\begin{bsmallmatrix}
					\eta_1\\
					\eta_2
				\end{bsmallmatrix}
				\in\RR^{\rDim_{\amplitude}+\rDim_{\shift}}} \frac12\norm{\res(t,\eta_1,\eta_2,\amplitude(t),\shift(t),u(t))}_{E(t)^{-\top}Q(t)^\top}^2
			\end{equation}
			for all $t\in \timeInterval$ and for any input signal $u\colon \RR_{\ge 0}\to\RR^m$ which admits a solution of \eqref{eq:nonlinearPHDAE_Q_factorizableMOR_stateEq}.
	\end{enumerate}
\end{theorem}

\begin{proof}
	The statement \eqref{itm:LTVPH_separableMOR_structurePreservation} is a special case of the upcoming \Cref{thm:nonlinearPHDAE_Q_separableMOR}.
	The proof of \eqref{itm:LTVPH_separableMOR_optimality} follows along the lines of the proof of \Cref{thm:linearPH_linearMOR}, where the major difference is that the first-order necessary optimality condition of the minimization problem \eqref{eq:pHROM_LTVFOM_separableMOR_optimality} for fixed $t\in\timeInterval$ is given by
	\begin{align*}
		&\Vs(\shift(t))^\top Q(t)^\top E(t)\left(\Vs(\shift(t))\eta_1+\widehat\Vs(\shift(t))(\amplitude(t))\eta_2\right)\\
		& = \Vs(\shift(t))^\top Q(t)^\top\left(\left((J(t)-R(t))Q(t)-K(t)\right)\Vs(\shift(t))\amplitude(t)-B(t)u(t)\right),\\[0.2cm]
		&\left(\widehat\Vs(\shift(t))\amplitude(t)\right)^\top Q(t)	^\top E(t)\left(\Vs(\shift(t))\eta_1+\widehat\Vs(\shift(t))(\amplitude(t))\eta_2\right)\\
		& = \left(\widehat\Vs(\shift(t))\amplitude(t)\right)^\top Q(t)^\top\left((J(t)-R(t))Q(t)-K(t)\right)\Vs(\shift(t))\amplitude(t)\\
		&\quad-\left(\widehat\Vs(\shift(t))\amplitude(t)\right)^\top Q(t)^\top B(t)u(t)
	\end{align*}
	and the corresponding Hessian is 
	\begin{equation*}
		\begin{bmatrix}
			\Vs(\shift(t))^\top Q(t)^\top E(t)\Vs(\shift(t)) 												& \Vs(\shift(t))^\top Q(t)^\top E(t)\widehat\Vs(\shift(t))\amplitude(t)\\
			\left(\widehat\Vs(\shift(t))\amplitude(t)\right)^\top Q(t)^\top E(t)\Vs(\shift(t)) 	& \left(\widehat\Vs(\shift(t))\amplitude(t)\right)^\top Q(t)^\top E(t)\widehat\Vs(\shift(t))\amplitude(t)
		\end{bmatrix}
		=\redE(t,\redState(t)).
	\end{equation*}
	Especially, we note that the Hessian depends neither on $\eta_1$ nor on $\eta_2$ and is symmetric and positive semi-definite.
	Consequently, the minimization problem \eqref{eq:pHROM_LTVFOM_separableMOR_optimality} is convex and the first-order necessary optimality condition is also sufficient, cf.~\cite{Nes04}.
\end{proof}

We proceed by considering structure-preserving \gls{mor} for nonlinear time-varying pH systems of the form \eqref{eq:nonlinearPHDAE_Q}--\eqref{eq:nonlinearPHDAEProperties_Q_new}.
For this purpose, we first point out that the residual minimization property of the \gls{rom} considered in \Cref{thm:LTVPH_separableMOR} relies on the assumption of $E^\top Q$ being pointwise symmetric and positive definite.
Since $E^\top Q$ coincides with the Hessian of the Hamiltonian with respect to the state, the pointwise positive definiteness of $E^\top Q$ corresponds to the assumption that for fixed $t\in\timeInterval$ the Hamiltonian is equivalent to a squared norm of the state, cf.~\Cref{rem:motivationForWeightedResNorm}.
Also when considering a nonlinear port-Hamiltonian \gls{fom} of the form \eqref{eq:nonlinearPHDAE_Q}--\eqref{eq:nonlinearPHDAEProperties_Q_new}, one may obtain a port-Hamiltonian \gls{rom} via weighted residual minimization, provided that $E^\top Q$ is pointwise symmetric and positive definite.
However, in contrast to the linear case, the matrix function $E^\top Q$ does in general not coincide with the Hessian of the Hamiltonian associated with the nonlinear \gls{ph} system \eqref{eq:nonlinearPHDAE_Q}.
Thus, assuming $E^\top Q$ to be pointwise positive definite in the context of the nonlinear \gls{ph} structure \eqref{eq:nonlinearPHDAE_Q}--\eqref{eq:nonlinearPHDAEProperties_Q_new} appears to be less natural than in the context of linear \gls{ph} systems.
For this reason, we only focus in the following on the structure preservation without addressing the question whether the \gls{rom} is optimal in some sense.

Similarly as in the case of a linear \gls{fom}, the \gls{rom} is constructed by enforcing the residual at $t\in\timeInterval$ to be orthogonal to the column space of $Q(t,\Vs(\shift(t))\amplitude(t))[\Vs(\shift(t))\; \,\widehat\Vs(\shift(t))\amplitude(t)]$.
The resulting \gls{rom} is of the form \eqref{eq:nonlinearPHDAE_Q_factorizableMOR} with $\redE,\redJ,\redR,\redQ,\redrph,\redB$ being defined via
\begin{subequations}
	\label{eq:pHROM_NTVFOM_separableMOR_coefficients}
	\begin{equation}
		\label{eq:pHROM_NTVFOM_separableMOR_coefficients_E}
		\begin{aligned}
			\redE(t,\redState) &\vcentcolon= 
			\begin{bmatrix}
				\redE_{11}(t,\redState) & \redE_{12}(t,\redState)\\
				\redE_{21}(t,\redState) & \redE_{22}(t,\redState)
			\end{bmatrix}
			,\\ 
			\redE_{11}(t,\redState) &\vcentcolon= \Vs(\shift)^\top Q(t,\Vs(\shift)\amplitude)^\top E(t,\Vs(\shift)\amplitude)\Vs(\shift)\in\RR^{\rDim_{\amplitude},\rDim_{\amplitude}},\\ 
			\redE_{12}(t,\redState) &\vcentcolon= \Vs(\shift)^\top Q(t,\Vs(\shift)\amplitude)^\top E(t,\Vs(\shift)\amplitude)\widehat\Vs(\shift)\amplitude\in\RR^{\rDim_{\amplitude},\rDim_{\shift}},\\ 
			\redE_{21}(t,\redState) &\vcentcolon=  \left(\widehat\Vs(\shift)\amplitude\right)^\top Q(t,\Vs(\shift)\amplitude)^\top E(t,\Vs(\shift)\amplitude)\Vs(\shift)\in\RR^{\rDim_{\amplitude},\rDim_{\shift}},\\
			\redE_{22}(t,\redState) &\vcentcolon= \left(\widehat\Vs(\shift)\amplitude\right)^\top Q(t,\Vs(\shift)\amplitude)^\top E(t,\Vs(\shift)\amplitude)\widehat\Vs(\shift)\amplitude\in\RR^{\rDim_{\shift},\rDim_{\shift}},\hspace{1.65cm}
		\end{aligned}
	\end{equation}		
	\vspace{0.2cm}
	\begin{equation}
		\begin{aligned}
			\redJ(t,\redState) &\vcentcolon= 
			\begin{bmatrix}
				\redJ_{11}(t,\redState) 	& -\redJ_{21}(t,\redState)^\top\\
				\redJ_{21}(t,\redState) 	& 0
			\end{bmatrix}
			,\\
			\redJ_{11}(t,\redState) &\vcentcolon= \Vs(\shift)^\top Q(t,\Vs(\shift)\amplitude)^\top J(t,\Vs(\shift)\amplitude)Q(t,\Vs(\shift)\amplitude)\Vs(\shift)\in\RR^{\rDim_{\amplitude},\rDim_{\amplitude}},\\ 
			\hspace{1.26cm}\redJ_{21}(t,\redState) &\vcentcolon= \left(\widehat\Vs(\shift)\amplitude\right)^\top Q(t,\Vs(\shift)\amplitude)^\top J(t,\Vs(\shift)\amplitude)Q(t,\Vs(\shift)\amplitude)\Vs(\shift)\in\RR^{\rDim_{\shift},\rDim_{\amplitude}},\hspace{0.2cm}
		\end{aligned}
	\end{equation}		
	\vspace{0.2cm}
	\begin{equation}
		\begin{aligned}
			\redR(t,\redState) &\vcentcolon= 
			\begin{bmatrix}
				\redR_{11}(t,\redState)	& \redR_{21}(t,\redState)^\top\\
				\redR_{21}(t,\redState) 	& \redR_{22}(t,\redState)
			\end{bmatrix}
			,\\ 
			\redR_{11}(t,\redState) &\vcentcolon= \Vs(\shift)^\top Q(t,\Vs(\shift)\amplitude)^\top R(t,\Vs(\shift)\amplitude)Q(t,\Vs(\shift)\amplitude)\Vs(\shift)\in\RR^{\rDim_{\amplitude},\rDim_{\amplitude}},\\ 
			\redR_{21}(t,\redState) &\vcentcolon=\left(\widehat\Vs(\shift)\amplitude\right)^\top Q(t,\Vs(\shift)\amplitude)^\top R(t,\Vs(\shift)\amplitude)Q(t,\Vs(\shift)\amplitude)\Vs(\shift)\in\RR^{\rDim_{\shift},\rDim_{\amplitude}},\\ 
			\hspace{1.2cm}\redR_{22}(t,\redState) &\vcentcolon= \left(\widehat\Vs(\shift)\amplitude\right)^\top Q(t,\Vs(\shift)\amplitude)^\top R(t,\Vs(\shift)\amplitude)Q(t,\Vs(\shift)\amplitude)\widehat\Vs(\shift)\amplitude\in\RR^{\rDim_{\shift},\rDim_{\shift}},
		\end{aligned}
	\end{equation}		
	\vspace{0.2cm}
	\begin{equation}
		\begin{aligned}
			\redrph(t,\redState) &\vcentcolon= 
			\begin{bmatrix}
				\redrph_1(t,\redState)\\
				\redrph_2(t,\redState)
			\end{bmatrix}
			,\\
			\redrph_1(t,\redState) &\vcentcolon= \Vs(\shift)^\top Q(t,\Vs(\shift)\amplitude)^\top \rph(t,\Vs(\shift)\amplitude)\in\RR^{\rDim_{\amplitude}},\\ 
			\redrph_2(t,\redState) &\vcentcolon= \left(\widehat\Vs(\shift)\amplitude\right)^\top Q(t,\Vs(\shift)\amplitude)^\top \rph(t,\Vs(\shift)\amplitude)\in\RR^{\rDim_{\shift}},\hspace{2.98cm}
		\end{aligned}
	\end{equation}		
	\vspace{0.2cm}
	\begin{equation}
		\begin{aligned}
			\redQ(t,\redState) \vcentcolon= 
			\begin{bmatrix}
				I_{\rDim_{\amplitude}} 	& 0\\
				0 									& 0
			\end{bmatrix}
			,\hspace{8.52cm}
		\end{aligned}
	\end{equation}		
	\vspace{0.2cm}
	\begin{equation}
		\begin{aligned}
			\redB(t,\redState) &\vcentcolon=
			\begin{bmatrix}
				\redB_1(t,\redState)\\
				\redB_2(t,\redState)
			\end{bmatrix}
			,\\ 
			\redB_1(t,\redState) &\vcentcolon= \Vs(\shift)^\top Q(t,\Vs(\shift)\amplitude)^\top B(t,\Vs(\shift)\amplitude)\in\RR^{\rDim_{\amplitude},m},\\ 
			\redB_2(t,\redState) &\vcentcolon= \left(\widehat\Vs(\shift)\amplitude\right)^\top Q(t,\Vs(\shift)\amplitude)^\top B(t,\Vs(\shift)\amplitude)\in\RR^{\rDim_{\shift},m}.\hspace{2.63cm}
		\end{aligned}
	\end{equation}
\end{subequations}
Here, we use again the notation from \eqref{eq:separableMOR_splitState} for the block components of $\redState$, while $\widehat\Vs$ is as defined in \eqref{eq:widehatVr_separableMOR}.
In the following theorem, it is stated that the \gls{rom} given by \eqref{eq:nonlinearPHDAE_Q_factorizableMOR} and \eqref{eq:pHROM_NTVFOM_separableMOR_coefficients} inherits the \gls{ph} structure of the corresponding \gls{fom}.

\begin{theorem}[Structure-preserving \gls{mor} for \eqref{eq:nonlinearPHDAE_Q} using a separable approximation ansatz]
	\label{thm:nonlinearPHDAE_Q_separableMOR}
	Consider the \gls{ph} system \eqref{eq:nonlinearPHDAE_Q} with $E,\rph,J,R,Q$ and the associated Hamiltonian $\ham$ satisfying \eqref{eq:nonlinearPHDAEProperties_Q_new}.
	Furthermore, let $\Vs\colon\RR^{\rDim_\shift}\to\RR^{\fulldim,\rDim_\amplitude}$ with $\rDim_\amplitude,\rDim_\shift\in\NN$ and $\rDim\vcentcolon=\rDim_\amplitude+\rDim_\shift\le\fulldim$ be continuously differentiable and consider the corresponding \gls{rom} \eqref{eq:nonlinearPHDAE_Q_factorizableMOR} with coefficients $\redE,\redJ,\redR,\redQ,\redrph,\redB$ as defined in \eqref{eq:pHROM_NTVFOM_separableMOR_coefficients}.
	Moreover, we consider the \gls{rom} Hamiltonian $\redHam\colon\RR_{\ge 0}\times\RR^{\rDim}\to\RR$ defined via $\redHam(t,\redState) \vcentcolon= \ham(t,\Vs(\shift)\amplitude)$.
	Then, $\redHam$ is continuously differentiable and the \gls{rom} coefficients satisfy \eqref{eq:nonlinearPHDAEROMProperties} for all $(t,\redState)\in\RR_{\ge 0}\times \RR^{\rDim}$.
\end{theorem}

\begin{proof}
	The proof follows along the lines of the proof of \Cref{thm:nonlinearPHDAE_Q_factorizableMOR}, where the major difference is the computation of the partial derivatives of $\redHam$.
	In particular, the partial derivatives with respect to $\amplitude$ and $\shift$ are given by
	\begin{align*}
		\partial_{\amplitude}\redHam(t,\redState) &= \partial_{\sol}\ham(t,\Vs(\shift)\amplitude)\Vs(\shift),\\
		\partial_{\shift}\redHam(t,\redState)\zeta &= \partial_{\sol}\ham(t,\Vs(\shift)\amplitude)\Vs'(\shift)(\zeta)\amplitude = \partial_{\sol}\ham(t,\Vs(\shift)\amplitude)\widehat\Vs(\shift)(\amplitude)\zeta
	\end{align*}
	for all $(t,\redState,\zeta)\in\RR_{\ge 0}\times \RR^{\rDim}\times \RR^{\rDim_{\shift}}$ and, hence, we obtain
	\begin{align*}
		\nabla_{\redState}\redHam(t,\redState) &= 
		\begin{bmatrix}
			\partial_{\amplitude}\redHam(t,\redState) & \partial_{\shift}\redHam(t,\redState)
		\end{bmatrix}
		^\top =
		\begin{bmatrix}
			\Vs(\shift) & \widehat\Vs(\shift)\amplitude
		\end{bmatrix}
		^\top \nabla_{\sol}\ham(t,\Vs(\shift)\amplitude)\\
		&= 
		\begin{bmatrix}
			\Vs(\shift) & \widehat\Vs(\shift)\amplitude
		\end{bmatrix}
		^\top E(t,\Vs(\shift)\amplitude)^\top Q(t,\Vs(\shift)\amplitude)\Vs(\shift)\amplitude = \redE(t,\redState)^\top \redQ(t,\redState)\redState
	\end{align*}
	for all $(t,\redState)\in\RR_{\ge 0}\times \RR^{\rDim}$.
	Furthermore, for the partial derivative of $\redHam$ with respect to $t$ we compute
	\begin{equation*}
		\partial_t\redHam(t,\redState) = \partial_t\ham(t,\Vs(\shift)\amplitude) = \amplitude^\top\Vs(\shift)^\top Q(t,\Vs(\shift)\amplitude)^\top\rph(t,\Vs(\shift)\amplitude) = \redState^\top \redQ(t,\redState)^\top\redrph(t,\redState)
	\end{equation*}
	for all $(t,\redState)\in\RR_{\ge 0}\times \RR^{\rDim}$.
\end{proof}

We close this section by discussing the stability of the \gls{rom} state equation \eqref{eq:nonlinearPHDAE_Q_factorizableMOR_stateEq} with $u=0$ and coefficients as in \eqref{eq:pHROM_NTVFOM_separableMOR_coefficients}.
To this end, we first note that it is not possible to obtain a stability result as in \Cref{cor:nonlinearPHDAE_Q_factorizableMOR_stability} for two reasons.
First, in \Cref{cor:nonlinearPHDAE_Q_factorizableMOR_stability} it is assumed that $\redE$ is pointwise nonsingular and this assumption is not possible to satisfy in case that $\redE$ is defined as in \eqref{eq:pHROM_NTVFOM_separableMOR_coefficients_E}, since then $\redE$ is singular for $\amplitude=0$.
Second, in the proof of \Cref{cor:nonlinearPHDAE_Q_factorizableMOR_stability} we have used that if the \gls{fom} Hamiltonian is equivalent to a squared norm of the state and $\Vr$ satisfies \eqref{eq:boundForSingularValsOfVr_timeDependentFactorizableMOR}, then also the \gls{rom} Hamiltonian is equivalent to a squared norm of the \gls{rom} state.
In contrast, even if $\Vs$ satisfies a singular value bound similar to \eqref{eq:boundForSingularValsOfVr_timeDependentFactorizableMOR}, we may only infer that the \gls{rom} Hamiltonian is equivalent to a squared norm of $\amplitude$, but not of the whole \gls{rom} state $\redState = [\amplitude^\top\;\,\shift^\top]^\top$.
Consequently, we focus in the following on the boundedness of $\amplitude$.

\begin{corollary}[Boundedness of part of the state in \eqref{eq:nonlinearPHDAE_Q_factorizableMOR_stateEq} with \eqref{eq:pHROM_NTVFOM_separableMOR_coefficients}]
	\label{cor:nonlinearPHDAE_Q_separableMOR_stability}
	Let the assumptions of \Cref{thm:nonlinearPHDAE_Q_separableMOR} be satisfied and let there additionally exist constants $\check{c}_1,\check{c}_2\in\RR_{>0}$ with
	\begin{equation}
		\label{eq:linearPH_separableMOR_singularValueBoundsForVOfp}
		\maxSingularVal(\Vs(\eta)) \le \check{c}_1\quad \text{and}\quad \minSingularVal(\Vs(\eta)) \ge \check{c}_2\quad \text{for all }\eta\in\RR^{\rDim_{\shift}}.
	\end{equation}
	Furthermore, let the \gls{fom} Hamiltonian $\ham$ satisfy condition~\eqref{itm:normEquivalence} in \Cref{def:Lyapunov} with $V=\ham$.
	Besides, let $\redState = [\amplitude^\top\,\;\shift^\top]^\top \in C^1(\timeInterval,\RR^{\rDim_{\amplitude}+\rDim_{\shift}})$ be a solution of the \gls{rom} state equation \eqref{eq:nonlinearPHDAE_Q_factorizableMOR_stateEq} with $u=0$ and coefficients as in \eqref{eq:pHROM_NTVFOM_separableMOR_coefficients} on the time interval $\timeInterval = [t_0,\tend]$ with $t_0\in\RR_{\ge 0}$ and $\tend\in\RR_{>t_0}$.
	Then, there exists a constant $c\in\RR_{>0}$ which is independent of $t_0$ and $\tend$ and satisfies
	\begin{equation*}
		\norm{\amplitude(t)} \le c\norm{\amplitude(t_0)} \quad \text{for all }t\in\timeInterval.
	\end{equation*}
\end{corollary}

\begin{proof}
	From the continuous differentiability of $\redState$ and $\redHam$, cf.~\Cref{thm:nonlinearPHDAE_Q_separableMOR}, we infer that the function $\redHamAux\colon \timeInterval\to \RR$ defined via $\redHamAux(t) \vcentcolon= \redHam(t,\redState(t))$ is continuously differentiable as well.
	Furthermore, due to the \gls{ph} structure of the \gls{rom} as stated in \Cref{thm:nonlinearPHDAE_Q_separableMOR}, we conclude that $\redHamAux$ satisfies the dissipation inequality
	\begin{equation*}
		\redHamAuxDot(t) \le 0\quad \text{for all }t\in\timeInterval,
	\end{equation*}
 	where we have used $u=0$.
 	Consequently, we obtain
	\begin{equation*}
		\redHam(t,\redState(t)) = \redHamAux(t) \le \redHamAux(t_0) = \redHam(t_0,\redState(t_0))\quad \text{for all }t\in\timeInterval.
	\end{equation*}
	Using this inequality, the singular value bounds \eqref{eq:linearPH_separableMOR_singularValueBoundsForVOfp}, and the assumption that $\ham$ satisfies condition~\eqref{itm:normEquivalence} in \Cref{def:Lyapunov} with constants $c_2,c_3\in\RR_{>0}$, we arrive at
	\begin{align*}
		\norm{\amplitude(t)}^2 &\le \frac1{\minSingularVal(\Vs(\shift(t)))^2}\norm{\Vs(\shift(t))\amplitude(t)}^2 \le \frac1{c_2\check{c}_2^2}\ham(t,\Vs(\shift(t))\amplitude(t))\\
		&= \frac1{c_2\check{c}_2^2}\redHam(t,\redState(t)) \le \frac1{c_2\check{c}_2^2}\redHam(t_0,\redState(t_0)) = \frac1{c_2\check{c}_2^2}\ham(t_0,\Vs(\shift(t_0))\amplitude(t_0)) \\
		&\le \frac{c_3}{c_2\check{c}_2^2}\norm{\Vs(\shift(t_0))\amplitude(t_0)}^2 \le \frac{c_3\maxSingularVal(\Vs(\shift(t_0)))^2}{c_2\check{c}_2^2}\norm{\amplitude(t_0)}^2\\
		&\le \frac{c_3\check{c}_1^2}{c_2\check{c}_2^2}\norm{\amplitude(t_0)}^2
	\end{align*}
	for all $t\in\timeInterval$, which yields the assertion.
\end{proof}

Having a closer look at the proof of \Cref{cor:nonlinearPHDAE_Q_separableMOR_stability}, we observe that the same argumentation may be used to infer the bound
\begin{equation*}
	\norm{\Vs(\shift(t))\amplitude(t)} \le \sqrt{\frac{c_3}{c_2}}\norm{\Vs(\shift(t_0))\amplitude(t_0)}\quad \text{for all }t\in\timeInterval
\end{equation*}
for the approximation $\Vs(\shift)\amplitude$ of the \gls{fom} state.
Especially, this bound also applies in the case where the singular values of $\Vs$ are not uniformly bounded as in \eqref{eq:linearPH_separableMOR_singularValueBoundsForVOfp}.
In the special case where the \gls{fom} is linear and time-invariant as in \eqref{eq:pHEQ}--\eqref{eq:pHEQProperties}, the ratio $c_3/c_2$ may be replaced by the condition number of $E^\top Q$, cf.~the end of \cref{sec:pH}.

\section{Numerical Examples}
\label{sec:numerics}

In this section we demonstrate the structure-preserving \gls{mor} framework presented in \cref{sec:pHMOR} by means of two numerical test cases.
A linear advection--diffusion equation with non-periodic boundary conditions is considered in \cref{sec:ADE} and we demonstrate the \gls{ph} structure of the \gls{fom} as well as the energy consistency of the \gls{rom}.
In \cref{sec:wildlandFire} we consider a wildland fire model which is given by a coupled nonlinear system of a \gls{pde} and an \gls{ode}.
Assuming periodic boundary conditions, we demonstrate that the \gls{fom} may be written as a dissipative Hamiltonian system, i.e., a \gls{ph} system without external ports.
Moreover, we compare a \gls{rom} based on the structure-preserving technique from \cref{sec:separableMOR} with a \gls{rom} obtained via a non-structure-preserving approach.

The time integration of the \glspl{rom} is performed using the implicit midpoint rule and the nonlinear systems occurring in each time step are solved using the \textsc{Matlab} function \texttt{fsolve} with default settings.
Furthermore, all relative error values reported in the following correspond to the relative error in a discretized $L^2(\timeInterval\times\Omega)$ norm, where $\timeInterval$ denotes the time interval and $\Omega$ the spatial domain.
The discretization of the time integral is performed using the composite trapezoidal rule, whereas the spatial $L^2$ norm is approximated via $\norm{\cdot}_{\Eh}$.
Here, $\Eh$ denotes the leading matrix of the left-hand side of the \gls{fom}, cf.~\eqref{eq:ADEh}--\eqref{eq:ADEh_coefficentMatrices} and \eqref{eq:wildfireFOM}.

\paragraph*{Code Availability}

The \textsc{Matlab} source code for the numerical examples can be obtained from the doi \href{https://doi.org/10.5281/zenodo.7613302}{10.5281/zenodo.7613302}.

\subsection{Advection--Diffusion Equation}
\label{sec:ADE}

The first test case is given by a linear advection--diffusion equation on the spatial domain $\Omega = (0,1)$ with mixed Robin--Neumann boundary conditions.
The corresponding governing equations read
\begin{equation}
	\label{eq:ADE}
	\left\{\begin{aligned}
		\partial_t\sol(t,\spaceVar) &= -c \partial_\spaceVar\sol(t,\spaceVar)+\diffusion\partial_{\spaceVar\spaceVar} \sol \left(t,\spaceVar\right) && \text{for all }(t,\spaceVar)\in\timeInterval\times\Omega,\\
		c\sol(t,0)-\diffusion\partial_{\spaceVar}\sol(t,0) &= cg(t) && \text{for all }t\in\timeInterval,\\
		\partial_\spaceVar \sol(t,1) &= 0 && \text{for all }t\in\timeInterval,\\
		\sol(0,\spaceVar) &= \sol_0(\spaceVar) && \text{for all }\spaceVar\in\Omega
	\end{aligned}\right.
\end{equation}
with unknown $\sol\colon \timeInterval\times\overline{\Omega}\to\RR$, advection speed $c\in\RR_{>0}$, diffusion coefficient $d\in\RR_{>0}$, Robin boundary value $g\colon \RR_{\ge 0}\to\RR$, and initial value $\sol_0\colon \overline{\Omega}\to\RR$.
The combination of Robin and Neumann boundary conditions as used in \eqref{eq:ADE} is sometimes referred to as Danckwerts boundary conditions, cf.~\cite{AguBPC22,Dan53}.

In order to discretize the initial-boundary value problem \eqref{eq:ADE} in space, we use a finite element scheme.
To this end, we first consider the following weak formulation:
Find $\sol\colon \timeInterval\times\overline{\Omega}\rightarrow \RR$ such that
\begin{enumerate}[(i)]
	\item for all $t\in\timeInterval$, $\sol(t,\cdot)$ is in $H^1(\Omega)$ and satisfies
		\begin{align*}
			\ip{\psi,\partial_t\sol(t,\cdot)}_{\LTwo{\Omega}} 	&= c\psi(0)g(t)-\frac{c}2\left(\ip{\psi,\partial_\spaceVar\sol(t,\cdot)}_{\LTwo{\Omega}}-\ip{\psi^\prime,\sol(t,\cdot)}_{\LTwo{\Omega}}\right)\\
			&\quad-\diffusion\ip{\psi^\prime,\partial_\spaceVar\sol(t,\cdot)}_{\LTwo{\Omega}}-\frac{c}2\left(\psi(1)\sol(t,1)+\psi(0)\sol(t,0)\right)
		\end{align*} 
		for all $\psi\in H^1(\Omega)$,
	\item for all $\spaceVar\in \Omega$, we have $\sol(0,\spaceVar)=\sol_0(\spaceVar)$.
\end{enumerate}
Based on this weak formulation, we use a standard Galerkin finite element scheme based on an equidistant mesh with mesh size $h=\frac{1}{\nMesh+1}$, $N\in\NN$, and piecewise linear ansatz and test functions.
The resulting semi-discretized system takes the form
\begin{equation}
	\label{eq:ADEh}
		\Eh\dot\sol_h(t) 	= (\Jh-\Rh)\solh(t)+\Bh u(t)\quad \text{for all }t\in\timeInterval,
\end{equation}
where $\solh\colon\timeInterval\to\RR^{\nMesh+2}$ contains the coefficients corresponding to the finite element ansatz functions, the input $u\colon \RR_{\ge 0}\to\RR$ is given by $u=g$, and $\Eh,\Jh,\Rh\in\RR^{\nMesh+2,\nMesh+2}$, $\Bh\in\RR^{\nMesh+2}$ are defined as
\begin{subequations}
	\label{eq:ADEh_coefficentMatrices}
	\begin{align}
		\label{eq:ADEh_EhAndBh}
		\Eh &\vcentcolon= \frac{h}{6}
		\begin{bmatrix}
			2 & 1 & 0 & \cdots & 0 & 0\\
			1 & 4 & 1 & \ddots & \vdots & \vdots\\
			0 & 1 & 4 & \ddots & 0 & 0\\
			\vdots & \ddots & \ddots & \ddots & 1 & 0\\
			0 & \cdots & 0 & 1 & 4 & 1\\
			0 & \cdots & 0 & 0 & 1 & 2
		\end{bmatrix}
		,\quad \Bh \vcentcolon= c
		\begin{bmatrix}
			1\\
			0\\
			0\\
			\vdots\\
			0
		\end{bmatrix}
		,\\[0.1cm]
		\Jh &\vcentcolon= -\frac{c}{2}\tridiag_{\nMesh+2}(-1,0,1),\\ 
		\Rh &\vcentcolon= \frac{\diffusion}{h}
		\begin{bmatrix}
			1 & -1 & 0 & \cdots & 0 & 0\\
			-1 & 2 & -1 & \ddots & \vdots & \vdots\\
			0 & -1 & 2 & \ddots & 0 & 0\\
			\vdots & \ddots & \ddots & \ddots & -1 & 0\\
			0 & \cdots & 0 & -1 & 2 & -1\\
			0 & \cdots & 0 & 0 & -1 & 1
		\end{bmatrix}
		+\frac{c}{2}\diag(1,0,\ldots,0,1).
	\end{align}
\end{subequations}
Here, $\Eh$ is symmetric and positive definite, $\Jh$ is skew-symmetric, and $\Rh$ is symmetric and positive semi-definite.
Consequently, \eqref{eq:ADEh} represents the state equation of a \gls{ph} system of the form \eqref{eq:pHEQ} with Hamiltonian $\hamh(\solh)=\frac{1}{2}\solh^\top\Eh\solh$.

For the following numerical experiments, we choose the \gls{pde} parameters as $c=1$ and $\diffusion=10^{-3}$, the final time as $\tend=1.2$, and the boundary and initial values as
\begin{equation}
	\label{eq:ADE_ICBC}
	\begin{aligned}
		g(t) &= u(t) = 
		\begin{cases}
			\frac12\exp\left(1-\frac1{1-(20(t-0.225))^2}\right), & \text{if }t\in (0.175,0.275),\\
			0, 																		& \text{otherwise},
		\end{cases}
		\\
		\sol_0(\spaceVar) &= 
		\begin{cases}
			\exp\left(1-\frac1{1-(20(\spaceVar-\frac12))^2}\right), & \text{if }\spaceVar\in (0.45,0.55),\\
			0, 																						& \text{otherwise}
		\end{cases}
	\end{aligned}
\end{equation}
for all $t\in\RR_{\ge 0}$ and $\xi\in\overline{\Omega}$, respectively.
Moreover, we divide the spatial domain into $\nMesh+1=1000$ equidistant intervals, which corresponds to a mesh size of $h=10^{-3}$. 
For the time discretization, we use the implicit midpoint rule with step size $10^{-3}$.
\Cref{fig:ADE_FOM} depicts the numerical solution by means of a pseudocolor plot.
We observe that the initial wave profile is transported to the right, while its shape and amplitude change due to the diffusion.
After a certain time, a second wave enters the computational domain via the left boundary and is also transported to the right.

\begin{figure}[h!]
	\begin{center}
		\includegraphics[width=10cm]{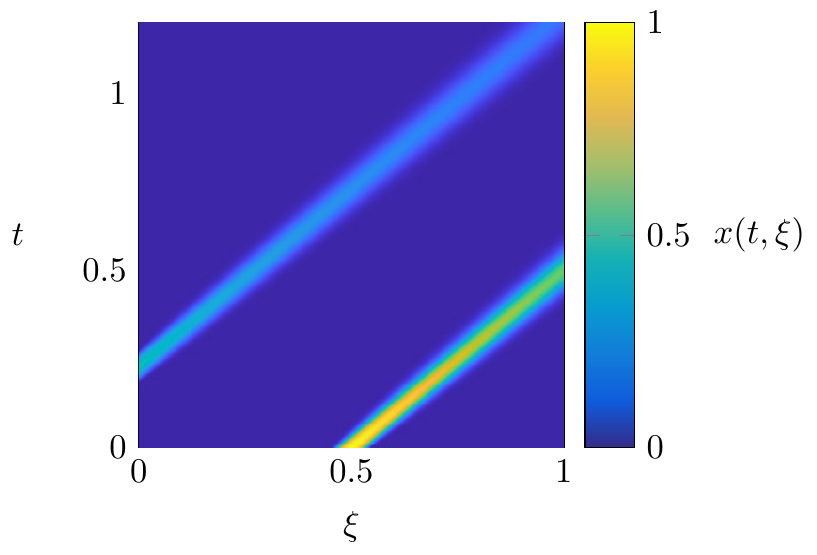}
	\end{center}
	\caption{Linear advection--diffusion equation: pseudocolor plot of the \gls{fom} solution.}
	\label{fig:ADE_FOM}
\end{figure}

In the following, we proceed similarly as in \cite{BlaSU20} and approximate the \gls{fom} state by a linear combination of transformed modes using an extended domain shift operator as transformation operator, cf.~\cite[sec.~7.2]{BlaSU20} for the details.
On the space-discrete level, the shift operation requires an interpolation scheme for obtaining values of the underlying continuous function in between the spatial grid points.
To this end, we employ cubic spline interpolation.
The resulting approximation ansatz takes the form
\begin{equation}
	\label{eq:ADE_approximationAnsatz}
	\solh(t) \approx \sum_{i=1}^{\rDim-1} \amplitude_i(t)\Text(\shift(t))\modes_i\quad \text{for all }t\in\timeInterval,
\end{equation}
where $\Text\colon \RR\to\RR^{\fulldim,\modeDim}$ with $\fulldim\vcentcolon=\nMesh+2$ is the discretized analogue of the extended domain shift operator, $\shift\colon\timeInterval\to\RR$ corresponds to the shift amount, $\modes_1,\ldots,\modes_{\rDim-1}\in\RR^{\modeDim}$ are the modes, $\amplitude_1,\ldots,\amplitude_{\rDim-1}\colon\timeInterval\to\RR$ the corresponding amplitudes, and $\modeDim$ the number of spatial grid points of the extended domain.
We note that \eqref{eq:ADE_approximationAnsatz} may be written as a separable ansatz of the form \eqref{eq:separableMORAnsatz} by defining $\Vs\colon \RR\to\RR^{\fulldim,\rDim-1}$ via
\begin{equation*}
	\Vs(\shift) \vcentcolon= \Text(\shift)
	\begin{bmatrix}
		\modes_1 & \cdots & \modes_{\rDim-1}
	\end{bmatrix}
	.
\end{equation*}
Based on the snapshot data depicted in \Cref{fig:ADE_FOM}, we determine $\rDim-1=3$ modes via the residual minimization approach presented in \cite{BlaSU20,SchRM19}.
The resulting relative offline error is $0.71\%$.
Afterwards, we use the structure-preserving projection framework detailed in \cref{sec:separableMOR} to obtain a corresponding port-Hamiltonian \gls{rom} of the form \eqref{eq:nonlinearPHDAE_Q_factorizableMOR} with coefficients as in \eqref{eq:pHROM_LTVFOM_separableMOR_coefficients}.
The resulting online error is $1.2\%$.
To demonstrate the energy consistency of the \gls{rom}, \Cref{fig:ADE_powerBalance} depicts the (discretized) time derivative of the \gls{rom} Hamiltonian $\redHam$ as well as the corresponding dissipation and supply rate at the midpoints of the discrete time intervals, cf.~\cref{sec:pH}.
Especially, we observe that the power balance \eqref{eq:nonlinearPHDAE_dissipationIneq} is approximately satisfied, as the graphs corresponding to $\frac{\diff\redHam}{\diff t}$ and $\redSupplyRate-\redDiss$ lie on top of each other.
The fact that the power balance is only approximately satisfied is illustrated in \Cref{tab:ADE_powerBalance}, where the corresponding mean and maximum errors are summarized for three different values of the time step size.
Especially, the results indicate that the error is mainly due to the time discretization, as the errors decrease with decreasing time step size.
We note that the implicit midpoint rule would yield a time-discrete system where the power balance is satisfied without any error, if the \gls{rom} Hamiltonian were a quadratic function of the \gls{rom} state, cf.~\cite{KotL19,MehM19}.
However, due to the nonlinear approximation ansatz this is not the case here, cf.~\eqref{eq:LTVPH_separableMOR_redHam}.

\begin{figure}[h!]
	\begin{center}
		\includegraphics[width=12cm]{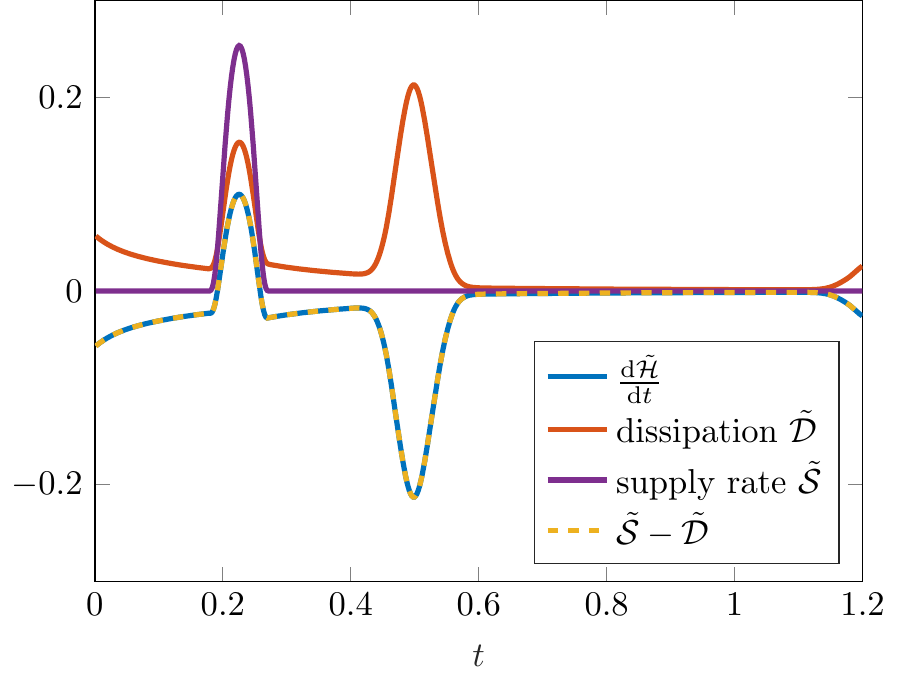}
	\end{center}
	\caption{Linear advection--diffusion equation: comparison of the temporal change of the \gls{rom} Hamiltonian $\redHam$ and the corresponding dissipation $\redDiss$ and supply rate $\redSupplyRate$.}
	\label{fig:ADE_powerBalance}
\end{figure}

\begin{table}[h!]
	\centering
	\caption{Linear advection--diffusion equation: Comparison of the maximum and mean errors in the \gls{rom} power balance for different time step sizes.}
	\label{tab:ADE_powerBalance}
	\begin{tabular}{lcc}
		\toprule
		time step size & maximum error & mean error\\\midrule
		$10^{-3}$ & $1.7\cdot 10^{-5}$ & $1.4\cdot 10^{-6}$ \\
		$5\cdot 10^{-4}$ & $3.9\cdot 10^{-6}$ & $4.1\cdot 10^{-7}$\\
		$2\cdot 10^{-4}$ & $6.9\cdot 10^{-7}$ & $6.3\cdot 10^{-8}$\\\bottomrule
	\end{tabular}
\end{table}

\subsection{Wildland Fire Model}
\label{sec:wildlandFire}

As second example, we consider a model which describes the dynamics of a wildland fire, cf.~\cite{BlaSU21b,ManBBCDKV08}.
The governing equations on a one-dimensional spatial domain $\Omega=(a,b)$ are given by
\begin{equation}
	\label{eq:wildfirePDE}
	\begin{aligned}
		\partial_t T &= k\partial_{\spaceVar\spaceVar}T-w\partial_\spaceVar T+\alpha (S\reactionRate(T,\arrhenius)-\gamma T),\\
		\partial_t S &= -\preExpFac S\reactionRate(T,\arrhenius),
	\end{aligned}
\end{equation}
where the unknowns are the relative temperature $T\colon \timeInterval\times \overline{\Omega}\to \RR$ and the supply mass fraction $S \colon \timeInterval\times \overline{\Omega}\to \RR$.
Furthermore, the constants $k,\alpha,\arrhenius,\gamma,\preExpFac\in\RR_{>0}$ and $w\in\RR$ are assumed to be given and $\reactionRate\colon \RR\times\RR\to \RR$ is defined via
\begin{equation*}
	\reactionRate(T,\arrhenius) \vcentcolon=
	\begin{cases}
		\exp\left(-\frac\arrhenius{T}\right), 	& \text{if }T>0,\\
		0, 													& \text{otherwise}.
	\end{cases}
\end{equation*}
For the physical meaning of these coefficients, we refer to \cite{BlaSU21b,ManBBCDKV08}.
Moreover, the system \eqref{eq:wildfirePDE} is closed via appropriate initial conditions and periodic boundary conditions, cf.~\cite{BlaSU21b}.

In contrast to the previous section, we follow \cite{BlaSU21b} for the spatial semi-discretization of \eqref{eq:wildfirePDE} and perform a central finite difference scheme based on an equidistant grid with grid size $h=\frac{b-a}{\nMesh+1}$, $\nMesh\in\NN$.
The resulting finite-dimensional system reads
\begin{equation}
	\label{eq:wildfireFOM}
	\begin{bmatrix}
		\dot{\sol}_1(t)\\
		\dot{\sol}_2(t)
	\end{bmatrix}
	= 
	\begin{bmatrix}
		kD_2-wD_1-\alpha\gamma I_{\nMesh+1} 	& \alpha\reactionRateMat(\sol_1(t),\arrhenius)\\
		0																& -\preExpFac\reactionRateMat(\sol_1(t),\arrhenius)
	\end{bmatrix}
	\begin{bmatrix}
		\sol_1(t)\\
		\sol_2(t)
	\end{bmatrix}
	,
\end{equation}
where $\sol_1,\sol_2\colon \timeInterval\to \RR^{\nMesh+1}$ correspond to approximations of $T$ and $S$ at the spatial grid points $ih$ for $i=1,\ldots,\nMesh+1$.
Moreover, $D_1 = -D_1^\top$ and $D_2=D_2^\top \le 0$ are finite difference approximations of the first and second spatial derivative, respectively, and the function $\reactionRateMat\colon \RR^{\nMesh+1}\times\RR\to \RR^{\nMesh+1, \nMesh+1}$ is given by
\begin{equation*}
	\reactionRateMat(\sol_1,\arrhenius) \vcentcolon= \diag\left(\reactionRate([\sol_1]_1,\arrhenius),\ldots,\reactionRate([\sol_1]_{\nMesh+1},\arrhenius)\right).
\end{equation*}
In the following, we demonstrate that the semi-discretized wildland fire model \eqref{eq:wildfireFOM} may be formulated as a dissipative Hamiltonian system.
To this end, we introduce $\eta \vcentcolon= \frac\alpha{4\gamma\preExpFac}$ and observe that \eqref{eq:wildfireFOM} may be written as
\begin{equation}
	\label{eq:wildfirePH}
	\begin{aligned}
		\begin{bmatrix}
			\dot{\sol}_1(t)\\
			\dot{\sol}_2(t)
		\end{bmatrix}
		= &\Bigg( \underbrace{
		\begin{bmatrix}
			-wD_1 	& 0\\
			0			& 0
		\end{bmatrix}
		}_{=\vcentcolon J_1}+\underbrace{\frac\alpha{2\eta}
		\begin{bmatrix}
			0 															& \reactionRateMat(\sol_1(t),\arrhenius)\\
			-\reactionRateMat(\sol_1(t),\arrhenius) 	& 0
		\end{bmatrix}
		}_{=\vcentcolon J_2(\sol_1(t))}\Bigg)
		\underbrace{
		\begin{bmatrix}
			I_{\nMesh+1} 	& 0\\
			0 						& \eta I_{\nMesh+1}
		\end{bmatrix}
		}_{=\vcentcolon Q}
		\begin{bmatrix}
			\sol_1(t)\\
			\sol_2(t)
		\end{bmatrix}
		\\
		&\quad-\Bigg(\underbrace{k
		\begin{bmatrix}
			-D_2 & 0\\
			0		& 0
		\end{bmatrix}
		}_{=\vcentcolon R_1}+\underbrace{
		\begin{bmatrix}
			\alpha\gamma I_{\nMesh+1} 									& -\frac\alpha{2\eta}\reactionRateMat(\sol_1(t),\arrhenius)\\
			-\frac\alpha{2\eta}\reactionRateMat(\sol_1(t),\arrhenius) 	& \frac\preExpFac\eta\reactionRateMat(\sol_1(t),\arrhenius)
		\end{bmatrix}
		}_{=\vcentcolon R_2(\sol_1(t))}\Bigg)
		\underbrace{
		\begin{bmatrix}
			I_{\nMesh+1} 	& 0\\
			0 						& \eta I_{\nMesh+1}
		\end{bmatrix}
		}_{=Q}
		\begin{bmatrix}
			\sol_1(t)\\
			\sol_2(t)
		\end{bmatrix}
		.
	\end{aligned}
\end{equation}
Here, we note that $Q$ is symmetric and positive definite since $\eta$ is positive and that $J_2$ and $R_2$ are pointwise skew-symmetric and symmetric, respectively.
Furthermore, since $D_1$ is skew-symmetric and $D_2$ is symmetric and negative semi-definite, we infer that $J_1$ is skew-symmetric and that $R_1$ is symmetric and positive semi-definite.
Thus, if we can show that additionally $R_2$ is pointwise positive semi-definite, we may conclude that \eqref{eq:wildfirePH} is a dissipative Hamiltonian system.
For this purpose, let $z = [p^\top\;\; q^\top]^\top\in\RR^{2(\nMesh+1)}$ with $p,q\in\RR^{\nMesh+1}$ and $u\in\RR^{\nMesh+1}$ be arbitrary.
Then, we obtain
\begin{align*}
	z^\top R_2(u)z 	&= \alpha\gamma p^\top p +\frac1\eta\sum_{i=1}^{\nMesh+1}\left(\preExpFac \reactionRate(u_i,\arrhenius)q_i^2-\alpha\reactionRate(u_i,\arrhenius)p_iq_i\right)\\
							&= \gamma\sum_{i=1}^{\nMesh+1}\Big(\underbrace{\alpha p_i^2-4\zeta\reactionRate(u_i,\arrhenius)p_iq_i+\frac{4\zeta^2}\alpha \reactionRate(u_i,\arrhenius)q_i^2}_{=\vcentcolon s_i}\Big).
\end{align*}
In the case where $u_i\le 0$ holds for some $i\in\lbrace 1,\ldots, \nMesh+1\rbrace$, we have $\reactionRate(u_i,\arrhenius)=0$ and, hence, $s_i\ge 0$.
Otherwise, we obtain
\begin{align*}
	s_i 	&= \alpha p_i^2-4\zeta\exp\left(-\frac\arrhenius{u_i}\right)p_iq_i+\frac{4\zeta^2}\alpha \exp\left(-\frac\arrhenius{u_i}\right)q_i^2\\
			&\ge \alpha p_i^2-4\zeta\exp\left(-\frac\arrhenius{u_i}\right)p_iq_i+\frac{4\zeta^2}\alpha\exp\left(-\frac{2\arrhenius}{u_i}\right)q_i^2\\
			&= \alpha p_i^2-4\zeta\exp\left(-\frac\arrhenius{u_i}\right)p_iq_i+\frac{4\zeta^2}\alpha \left(\exp\left(-\frac\arrhenius{u_i}\right)\right)^2q_i^2\\
			&= \left(\sqrt\alpha p_i-\frac{2\zeta}{\sqrt\alpha}\exp\left(-\frac\arrhenius{u_i}\right)q_i\right)^2 \ge 0.
\end{align*}
Consequently, $R_2$ is pointwise symmetric and positive semi-definite and, hence, \eqref{eq:wildfirePH} is a dissipative Hamiltonian formulation of \eqref{eq:wildfireFOM} with $J \vcentcolon= J_1+J_2$ and $R\vcentcolon= R_1+R_2$.

For the following numerical experiments, we choose the physical and discretization parameters as detailed in \cite[sec.~5.4]{BlaSU21b}.
The resulting snapshots are depicted in \Cref{fig:wildlandFire_FOM}.
Especially, we observe two traveling waves propagating through the computational domain.

\begin{figure}[h!]
	\begin{center}
		\includegraphics[width=6.3in]{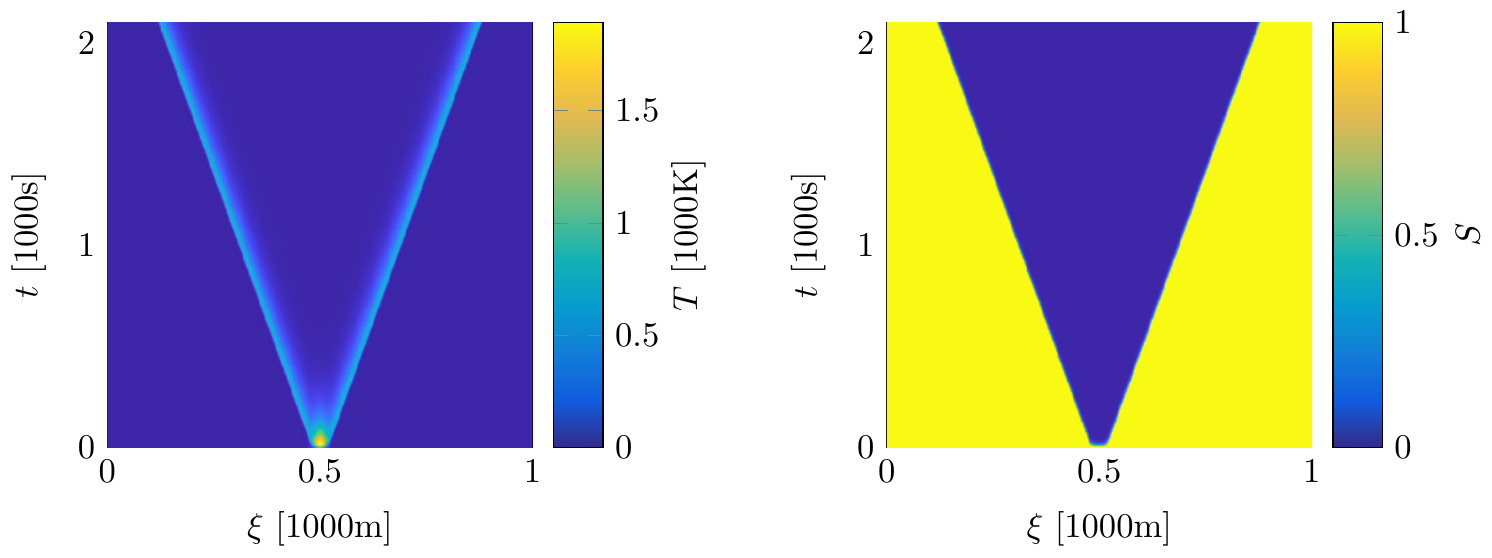}
	\end{center}
	\caption{Wildland fire model: pseudocolor plot of the temperature (left) and supply mass fraction (right).}
	\label{fig:wildlandFire_FOM}
\end{figure}

For the model reduction, we follow a slightly different approach than in \cite{BlaSU21b} since our main focus is on demonstrating the structure preservation rather than on the accuracy and evaluation time of the \gls{rom}.
Similar to \cref{sec:ADE}, we approximate the \gls{fom} state by a linear combination of two transformed modes, one for each traveling wave.
However, instead of an extended domain shift operator we use a periodic shift operator, which is again discretized using cubic spline interpolation.
The corresponding approximation ansatz reads
\begin{equation}
	\label{eq:wildlandFire_approximationAnsatz}
	\begin{bmatrix}
		\sol_1(t)\\
		\sol_2(t)
	\end{bmatrix}
	\approx
	\sum_{i=1}^2\amplitude_{i}(t)
	\begin{bmatrix}
		\Tper(\shift_i(t)) 	& 0\\
		0					& \Tper(\shift_i(t))
	\end{bmatrix}		
	\begin{bmatrix}
		\modes_{i,\mathrm{T}}\\
		\modes_{i,\mathrm{S}}
	\end{bmatrix}
	,
\end{equation}
where $\Tper\colon \RR\to\RR^{\fulldim,\fulldim}$ with $\fulldim\vcentcolon=\nMesh+1$ is the discretized analogue of the periodic shift operator and $\modes_{i,\mathrm{T}}$ and $\modes_{i,\mathrm{S}}$ denote the temperature and supply mass fraction block component of the $i$th mode, respectively, for $i=1,2$.
Similarly as in \cref{sec:ADE}, \eqref{eq:wildlandFire_approximationAnsatz} may be written as a separable ansatz of the form \eqref{eq:separableMORAnsatz}.

As in the previous subsection, the modes are determined via residual minimization yielding a relative offline error of $13\%$.
As demonstrated in \cite{BlaSU21b}, the error may be significantly reduced by increasing the mode numbers and introducing a suitable time interval splitting.
However, for illustrating the structure preservation, the relatively coarse approximation based on two shifted modes is sufficient here.
In the following, we compare two different ROMs which are both integrated in time using the implicit midpoint rule.
The first one is based on the nonlinear Galerkin approach from \cite[sec.~5]{BlaSU20} and enforces the residual at $t\in\timeInterval$ to be orthogonal to the column space of $[\Vs(\shift(t))\; \,\widehat\Vs(\shift(t))\amplitude(t)]$.
On the other hand, the second \gls{rom} is based on the structure-preserving projection framework outlined in \cref{sec:separableMOR}, i.e., it enforces the residual at $t\in\timeInterval$ to be orthogonal to the column space of $Q[\Vs(\shift(t))\; \,\widehat\Vs(\shift(t))\amplitude(t)]$ with $Q$ as in \eqref{eq:wildfirePH}.
The corresponding online approximations of the temperature field are compared in \Cref{fig:wildlandFire_ROM_pcolorPlots}.
The non-structure-preserving \gls{rom} based on the Galerkin approach yields a solution where the temperature rapidly decreases to zero and the fire goes out before traveling combustion waves may develop.
On the other-hand, the approximation obtained by the structure-preserving \gls{rom} reveals traveling combustion waves, although the flame speeds are significantly smaller than the ones obtained by the corresponding FOM.
Furthermore, \Cref{fig:wildfire_ROMPowerBalances} illustrates that the unsatisfactory solution obtained by the non-structure-preserving \gls{rom} is accompanied by an energy inconsistency.
Especially, at the beginning of the time interval, the decline of the Hamiltonian is much greater than the corresponding dissipation, i.e., the power balance is clearly violated.
Since the ROM Hamiltonian is a squared function of the amplitudes $\alpha$, cf.~\eqref{eq:LTVPH_separableMOR_redHam}, this rapid decline is also reflected in the abrupt temperature decrease observed in \Cref{fig:wildlandFire_ROM_pcolorPlots}.
On the other hand, the structure-preserving MOR approach yields an energy-consistent ROM as illustrated in \Cref{fig:wildfire_ROMPowerBalances}, right.

\begin{figure}[h!]
	\begin{center}
		\includegraphics[width=6.3in]{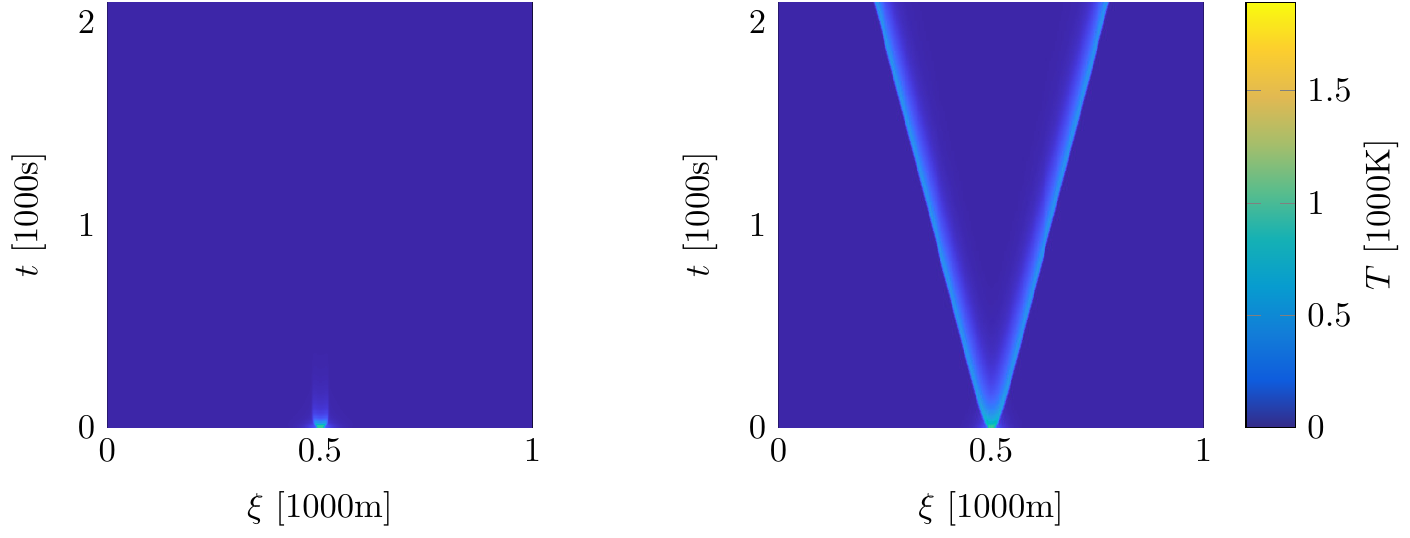}
	\end{center}
	\caption{Wildland fire model: pseudocolor plots of the temperature approximations using the non-structure-preserving ROM (left) and the structure-preserving one (right).}
	\label{fig:wildlandFire_ROM_pcolorPlots}
\end{figure}

\begin{figure}[h!]
	\begin{center}
		\includegraphics[width=6.3in]{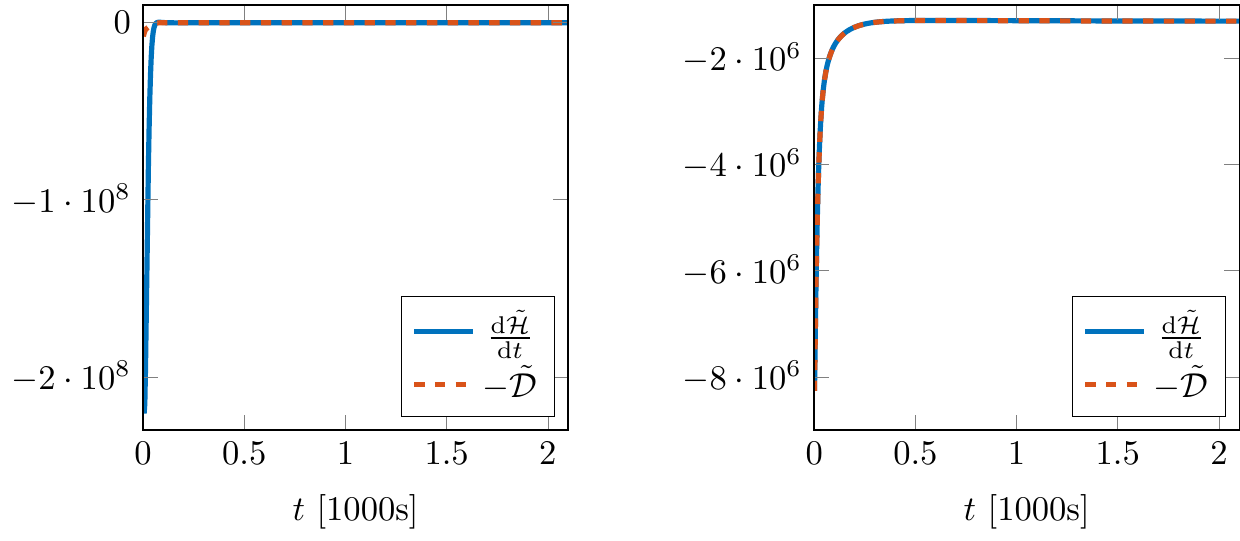}
	\end{center}
	\caption{Wildland fire model: comparison of the temporal change of the \gls{rom} Hamiltonian $\redHam$ and the corresponding negative dissipation $-\redDiss$ using the non-structure-preserving ROM (left) and the structure-preserving one (right).}
	\label{fig:wildfire_ROMPowerBalances}
\end{figure}

The results addressed in the previous paragraph demonstrate that the structure-preserving \gls{mor} approach does not only ensure port-Hamiltonian \glspl{rom}, but it may sometimes also lead to a gain in accuracy.
However, at this point we emphasize that there is in general no guarantee that this is the case and in most of our numerical experiments the accuracies of the structure-preserving and the non-structure-preserving ROMs have been comparable.
The theory from \cref{sec:separableMOR} only ensures the energy consistency of the \gls{rom}, but it does not include any statements about the accuracy in comparison to the \gls{fom}.

\section{Conclusion}
\label{sec:conclusion}

In this paper, we introduce a structure-preserving model order reduction (MOR) framework for port-Hamiltonian (pH) systems based on a special class of nonlinear approximation ansatzes.
In particular, we consider ansatzes where the approximation of the full-order state may be factorized as the product of a matrix, which may depend on the reduced-order model (ROM) state, and the \gls{rom} state itself.
Such ansatzes are for instance relevant in the context of transport-dominated systems which are challenging for classical methods based on linear subspace approximations.
Based on the considered class of ansatzes, we demonstrate how to obtain a port-Hamiltonian \gls{rom} via projection, provided that the corresponding full-order model (FOM) has a special \gls{ph} structure, which includes linear as well as a wide range of nonlinear \gls{ph} systems.
Moreover, for the special case that the \gls{fom} is linear, we identify a subclass of nonlinear ansatzes yielding \glspl{rom} which are not only \gls{ph}, but also optimal in the sense that the derivative of the \gls{rom} state minimizes a certain weighted norm of the residual.
In addition, we provide sufficient conditions for the \gls{rom} state equation with vanishing input signal to be stable.
Finally, the theoretical findings are illustrated by means of a linear advection--diffusion problem with non-periodic boundary conditions and a nonlinear reaction--diffusion system modeling the spread of wildland fires.

While we have considered a special class of nonlinear approximation ansatzes, an interesting future research direction is to derive structure-preserving \gls{mor} schemes based on more general nonlinear ansatzes.
This would especially allow to obtain port-Hamiltonian \glspl{rom} via projection onto nonlinear manifolds which are parametrized by artificial neural networks.
Another question not addressed in this manuscript is the preservation of algebraic constraints in cases where the \gls{fom} is given by a \gls{ph} system of differential--algebraic equations.
While corresponding approaches already exist in the context of classical \gls{mor} based on linear ansatzes as mentioned in \Cref{rem:algebraicConstraints}, this is still an open problem in the context of nonlinear ansatzes.
 
\paragraph{Acknowledgments}

This work was supported by the Deutsche Forschungsgemeinschaft (DFG, German Research Foundation) Collaborative Research Center Transregio 96 \emph{Thermo-energetic design of machine tools -- A systemic approach to solve the conflict between power efficiency, accuracy and productivity demonstrated at the example of machining production}, project number 174223256.
In addition, I would like to thank Volker Mehrmann and Riccardo Morandin from TU Berlin for helpful discussions.

\bibliographystyle{plain}
\bibliography{Refs}

\begin{thebibliography}{100}

\bibitem{AguBPC22}
L.~Agud~Albesa, M.~Boix~Garc\'ia, M.~L. Pla~Ferrando, and S.~C.
  Cardona~Navarrete.
\newblock A study about the solution of convection--diffusion--reaction
  equation with {D}anckwerts boundary conditions by analytical, method of lines
  and {C}rank--{N}icholson techniques.
\newblock {\em Math. Methods Appl. Sci.}, 2022.

\bibitem{AllK17}
A.~Alla and J.~N. Kutz.
\newblock Nonlinear model order reduction via dynamic mode decomposition.
\newblock {\em SIAM J. Sci. Comput.}, 39(5):B778--B796, 2017.

\bibitem{AltMU21}
R.~Altmann, V.~Mehrmann, and B.~Unger.
\newblock Port-{H}amiltonian formulations of poroelastic network models.
\newblock {\em Math. Comput. Model. Dyn. Syst.}, 27(1):429--452, 2021.

\bibitem{AltS17}
R.~Altmann and P.~Schulze.
\newblock A port-{H}amiltonian formulation of the {N}avier--{S}tokes equations
  for reactive flows.
\newblock {\em Systems Control Lett.}, 100:51--55, 2017.

\bibitem{AndF22a}
W.~Anderson and M.~Farazmand.
\newblock Evolution of nonlinear reduced-order solutions for {PDE}s with
  conserved quantities.
\newblock {\em SIAM J. Sci. Comput.}, 44(1):A176--A197, 2022.

\bibitem{AndF22b}
W.~Anderson and M.~Farazmand.
\newblock Shape-morphing reduced-order models for nonlinear {S}chr\"odinger
  equations.
\newblock {\em Nonlinear Dyn.}, 108:2889--2902, 2022.

\bibitem{Ant05b}
A.~C. Antoulas.
\newblock {\em Approximation of Large-Scale Dynamical Systems}.
\newblock Society for Industrial and Applied Mathematics, Philadelphia, PA,
  USA, 2005.

\bibitem{Ant05a}
A.~C. Antoulas.
\newblock A new result on passivity preserving model reduction.
\newblock {\em Systems Control Lett.}, 54:361--374, 2005.

\bibitem{AntBG20}
A.~C. Antoulas, C.~A. Beattie, and S.~G\"u\u{g}ercin.
\newblock {\em Interpolatory Methods for Model Reduction}.
\newblock Society for Industrial and Applied Mathematics, Philadelphia, PA,
  USA, 2020.

\bibitem{AstWWB08}
P.~Astrid, S.~Weiland, K.~Willcox, and T.~Backx.
\newblock Missing point estimation in models described by proper orthogonal
  decomposition.
\newblock {\em IEEE Trans. Automat. Control}, 53(10):2237--2251, 2008.

\bibitem{BanSAZISW21}
H.~Bansal, P.~Schulze, M.~H. Abbasi, H.~Zwart, L.~Iapichino, W.~H.~A.
  Schilders, and N.~van~de Wouw.
\newblock Port-{H}amiltonian formulation of two-phase flow models.
\newblock {\em Systems Control Lett.}, 149:104881, 2021.

\bibitem{BarF22}
J.~Barnett and C.~Farhat.
\newblock Quadratic approximation manifold for mitigating the {K}olmogorov
  barrier in nonlinear projection-based model order reduction.
\newblock {\em J. Comput. Phys.}, 464:111348, 2022.

\bibitem{BarMNP04}
M.~Barrault, Y.~Maday, N.~C. Nguyen, and A.~T. Patera.
\newblock An `empirical interpolation' method: application to efficient
  reduced-basis discretization of partial differential equations.
\newblock {\em C. R. Math. Acad. Sci. Paris}, 339:667--672, 2004.

\bibitem{BeaGM22}
C.~Beattie, S.~Gugercin, and V.~Mehrmann.
\newblock Structure-preserving interpolatory model reduction for
  port-{H}amiltonian differential-algebraic systems.
\newblock In C.~Beattie, P.~Benner, M.~Embree, S.~Gugercin, and S.~Lefteriu,
  editors, {\em Realization and Model Reduction of Dynamical Systems: A
  Festschrift in Honor of the 70th Birthday of Thanos Antoulas}, pages
  235--254. Springer Nature Switzerland, Cham, Switzerland, 2022.

\bibitem{BeaMXZ18}
C.~Beattie, V.~Mehrmann, H.~Xu, and H.~Zwart.
\newblock Linear port-{H}amiltonian descriptor systems.
\newblock {\em Math. Control Signals Systems}, 30:17, 2018.

\bibitem{BenB17}
P.~Benner and T.~Breiten.
\newblock Model order reduction based on system balancing.
\newblock In P.~Benner, M.~Ohlberger, A.~Cohen, and K.~Willcox, editors, {\em
  Model Reduction and Approximation}, pages 261--295. Society for Industrial
  and Applied Mathematics, Philadelphia, PA, USA, 2017.

\bibitem{BenF21}
P.~Benner and L.~Feng.
\newblock Model order reduction based on moment-matching.
\newblock In P.~Benner, S.~Grivet-Talocia, A.~Quarteroni, G.~Rozza,
  W.~Schilders, and L.~Miguel~Silveira, editors, {\em Model Order Reduction --
  Volume 1: System- and Data-Driven Methods and Algorithms}, pages 57--96. De
  Gruyter, Berlin, Germany, 2021.

\bibitem{BenOCW17}
P.~Benner, M.~Ohlberger, A.~Cohen, and K.~Willcox.
\newblock {\em Model Reduction and Approximation}.
\newblock Society for Industrial and Applied Mathematics, Philadelphia, PA,
  USA, 2017.

\bibitem{BenS17}
P.~Benner and T.~Stykel.
\newblock Model order reduction for differential-algebraic equations: A survey.
\newblock In A.~Ilchmann and T.~Reis, editors, {\em Surveys in
  Differential-Algebraic Equations IV}, pages 107--160. Springer International
  Publishing, Cham, Switzerland, 2017.

\bibitem{BerHL93}
G.~Berkooz, P.~Holmes, and J.~L. Lumley.
\newblock The proper orthogonal decomposition in the analysis of turbulent
  flows.
\newblock {\em Annu. Rev. Fluid Mech.}, 25:539--575, 1993.

\bibitem{BlaSU20}
F.~Black, P.~Schulze, and B.~Unger.
\newblock Projection-based model reduction with dynamically transformed modes.
\newblock {\em ESAIM: Math. Model. Numer. Anal.}, 54:2011--2043, 2020.

\bibitem{BlaSU21b}
F.~Black, P.~Schulze, and B.~Unger.
\newblock Efficient wildland fire simulation via nonlinear model order
  reduction.
\newblock {\em Fluids}, 6(8):280, 2021.

\bibitem{BorSF21}
P.~Borja, J.~M.~A. Scherpen, and K.~Fujimoto.
\newblock Extended balancing of continuous {LTI} systems: a
  structure-preserving approach.
\newblock {\em IEEE Trans. Automat. Control}, 2021.

\bibitem{BreMS22}
T.~Breiten, R.~Morandin, and P.~Schulze.
\newblock Error bounds for port-{H}amiltonian model and controller reduction
  based on system balancing.
\newblock {\em Comput. Math. with Appl.}, 116:100--115, 2022.

\bibitem{BreU22}
T.~Breiten and B.~Unger.
\newblock Passivity preserving model reduction via spectral factorization.
\newblock {\em Automatica}, 142:110368, 2022.

\bibitem{BruAPM19a}
A.~Brugnoli, D.~Alazard, V.~Pommier-Budinger, and D.~Matignon.
\newblock Port-{H}amiltonian formulation and symplectic discretization of plate
  models part {I}: {M}indlin model for thick plates.
\newblock {\em Appl. Math. Model.}, 75:940--960, 2019.

\bibitem{BruAPM19b}
A.~Brugnoli, D.~Alazard, V.~Pommier-Budinger, and D.~Matignon.
\newblock Port-{H}amiltonian formulation and symplectic discretization of plate
  models part {II}: {K}irchhoff model for thin plates.
\newblock {\em Appl. Math. Model.}, 75:961--981, 2019.

\bibitem{ByrIW91}
C.~I. Byrnes, A.~Isidori, and J.~C. Willems.
\newblock Passivity, feedback equivalence, and the global stabilization of
  minimum phase nonlinear systems.
\newblock {\em IEEE Trans. Automat. Control}, 36(11):1228--1240, 1991.

\bibitem{CagMS19}
N.~Cagniart, Y.~Maday, and B.~Stamm.
\newblock Model order reduction for problems with large convection effects.
\newblock In B.~N. Chetverushkin, W.~Fitzgibbon, Y.~A. Kuznetsov,
  P.~Neittaanm{\"a}ki, J.~Periaux, and O.~Pironneau, editors, {\em
  Contributions to Partial Differential Equations and Applications}, pages
  131--150. Springer International Publishing, Cham, Switzerland, 2019.

\bibitem{CarMP17}
F.~L. Cardoso-Ribeiro, D.~Matignon, and V.~Pommier-Budinger.
\newblock A port-{H}amiltonian model of liquid sloshing in moving containers
  and application to a fluid-structure system.
\newblock {\em J. Fluids Struct.}, 69:402--427, 2017.

\bibitem{CarBA17}
K.~Carlberg, M.~Barone, and H.~Antil.
\newblock Galerkin v{.}\ least-squares {P}etrov--{G}alerkin projection in
  nonlinear model reduction.
\newblock {\em J. Comput. Phys.}, 330:693--734, 2017.

\bibitem{CarBF11}
K.~Carlberg, C.~Bou-Mosleh, and C.~Farhat.
\newblock Efficient non-linear model reduction via a least-squares
  {P}etrov--{G}alerkin projection and compressive tensor approximations.
\newblock {\em Int. J. Numer. Methods Eng.}, 86:155--181, 2011.

\bibitem{CasLE12}
R.~Casta\~{n}\'e Selga, B.~Lohmann, and R.~Eid.
\newblock Stability preservation in projection-based model order reduction of
  large scale systems.
\newblock {\em Eur. J. Control}, 18(2):122--132, 2012.

\bibitem{ChaBG16}
S.~Chaturantabut, C.~Beattie, and S.~Gugercin.
\newblock Structure-preserving model reduction for nonlinear port-{H}amiltonian
  systems.
\newblock {\em SIAM J. Sci. Comput.}, 38(5):B837--B865, 2016.

\bibitem{ChaS10}
S.~Chaturantabut and D.~C. Sorensen.
\newblock Nonlinear model reduction via discrete empirical interpolation.
\newblock {\em SIAM J. Sci. Comput.}, 32(5):2737--2764, 2010.

\bibitem{CheS21}
X.~Cheng and J.~M.~A. Scherpen.
\newblock Model reduction methods for complex network systems.
\newblock {\em Annu. Rev. Control Robot. Auton. Syst.}, 4:425--453, 2021.

\bibitem{Dan53}
P.~V. Danckwerts.
\newblock Continuous flow systems: distribution of residence times.
\newblock {\em Chem. Eng. Sci.}, 2(1):1--13, 1953.

\bibitem{DuiMSB09}
V.~Duindam, A.~Macchelli, S.~Stramigioli, and H.~Bruyninckx.
\newblock {\em Modeling and Control of Complex Physical Systems}.
\newblock Springer-Verlag Berlin Heidelberg, Germany, 2009.

\bibitem{EggKLMM18}
H.~Egger, T.~Kugler, B.~Liljegren-Sailer, N.~Marheineke, and V.~Mehrmann.
\newblock On structure-preserving model reduction for damped wave propagation
  in transport networks.
\newblock {\em SIAM J. Sci. Comput.}, 40(1):A331--A365, 2018.

\bibitem{FalH17}
A.~Falaize and T.~H\'elie.
\newblock Passive simulation of the nonlinear port-{H}amiltonian modeling of a
  {R}hodes piano.
\newblock {\em J. Sound Vib.}, 390:289--309, 2017.

\bibitem{Far08}
J.~Faraut.
\newblock {\em Analysis on Lie Groups: An Introduction}.
\newblock Cambridge University Press, Cambridge, UK, 2008.

\bibitem{FarACC14}
C.~Farhat, P.~Avery, T.~Chapman, and J.~Cortial.
\newblock Dimensional reduction of nonlinear finite element dynamic models with
  finite rotations and energy-based mesh sampling and weighting for
  computational efficiency.
\newblock {\em Int. J. Numer. Methods Eng.}, 98(9):625--662, 2014.

\bibitem{FiaZOSS13}
S.~Fiaz, D.~Zonetti, R.~Ortega, J.~M.~A. Scherpen, and A.~J. van~der Schaft.
\newblock A port-{H}amiltonian approach to power network modeling and analysis.
\newblock {\em Eur. J. Control}, 19(6):477--485, 2013.

\bibitem{FreDM21}
S.~Fresca, L.~Ded\'{e}, and A.~Manzoni.
\newblock A comprehensive deep learning-based approach to reduced order
  modeling of nonlinear time-dependent parametrized {PDE}s.
\newblock {\em J. Sci. Comput.}, 87:61, 2021.

\bibitem{GerHRS21}
H.~Gernandt, F.~E. Haller, T.~Reis, and A.~J. van~der Schaft.
\newblock Port-{H}amiltonian formulation of nonlinear electrical circuits.
\newblock {\em J. Geom. Phys.}, 159:103959, 2021.

\bibitem{GifWPL14}
M.~Giftthaler, T.~Wolf, H.~K.~F. Panzer, and B.~Lohmann.
\newblock Parametric model order reduction of port-{H}amiltonian systems by
  matrix interpolation.
\newblock {\em Automatisierungstechnik}, 62(9):619--628, 2014.

\bibitem{GilMS18}
N.~Gillis, V.~Mehrmann, and P.~Sharma.
\newblock Computing the nearest stable matrix pairs.
\newblock {\em Numer. Linear Algebra Appl.}, 25(5):e2153, 2018.

\bibitem{GlaMM98}
S.~Glavaski, J.~E. Marsden, and R.~M. Murray.
\newblock Model reduction, centering, and the {K}arhunen--{L}oeve expansion.
\newblock In {\em Proceedings of the 37th IEEE Conference on Decision and
  Control}, pages 2071--2076, Tampa, FL, USA, 1998.

\bibitem{GolP03}
G.~Golub and V.~Pereyra.
\newblock Separable nonlinear least squares: the variable projection method and
  its applications.
\newblock {\em Inverse Probl.}, 19:R1--R26, 2003.

\bibitem{Gur08}
C.~Gu and J.~Roychowdhury.
\newblock Model reduction via projection onto nonlinear manifolds, with
  applications to analog circuits and biochemical systems.
\newblock In {\em IEEE/ACM International Conference on Computer Aided Design
  (ICCAD)}, pages 85--92, San Jose, CA, USA, 2008.

\bibitem{GubV17}
M.~Gubisch and S.~Volkwein.
\newblock Proper orthogonal decomposition for linear-quadratic optimal control.
\newblock In P.~Benner, M.~Ohlberger, A.~Cohen, and K.~Willcox, editors, {\em
  Model Reduction and Approximation}, pages 3--63. Society for Industrial and
  Applied Mathematics, Philadelphia, PA, USA, 2017.

\bibitem{GugAB08}
S.~Gugercin, A.~C. Antoulas, and C.~Beattie.
\newblock $\mathcal{H}_2$ model reduction for large-scale linear dynamical
  systems.
\newblock {\em SIAM J. Matrix Anal. Appl.}, 30(2):609--638, 2008.

\bibitem{GugPBS12}
S.~Gugercin, R.~V. Polyuga, C.~Beattie, and A.~van~der Schaft.
\newblock Structure-preserving tangential interpolation for model reduction of
  port-{H}amiltonian systems.
\newblock {\em Automatica}, 48(9):1963--1974, 2012.

\bibitem{HarVS10}
C.~Hartmann, V.-M. Vulcanov, and C.~Sch\"{u}tte.
\newblock Balanced truncation of linear second-order systems: a {H}amiltonian
  approach.
\newblock {\em Multiscale Model. Simul.}, 8(4):1348--1367, 2010.

\bibitem{HauMM19b}
S.-A. Hauschild, N.~Marheineke, and V.~Mehrmann.
\newblock Model reduction techniques for linear constant coefficient
  port-{H}amiltonian differential-algebraic systems.
\newblock {\em Control Cybern.}, 48(1):125--152, 2019.

\bibitem{HesRS16}
J.~S. Hesthaven, G.~Rozza, and B.~Stamm.
\newblock {\em Certified Reduced Basis Methods for Parametrized Partial
  Differential Equations}.
\newblock Springer International Publishing, Cham, Switzerland, 2016.

\bibitem{HoaCJL11}
H.~Hoang, F.~Couenne, C.~Jallut, and Y.~Le~Gorrec.
\newblock The port {Hamiltonian} approach to modeling and control of continuous
  stirred tank reactors.
\newblock {\em Journal of Process Control}, 21(10):1449--1458, 2011.

\bibitem{IonA13a}
T.~C. Ionescu and A.~Astolfi.
\newblock Families of moment matching based, structure preserving
  approximations for linear port {Hamiltonian} systems.
\newblock {\em Automatica}, 49(8):2424--2434, 2013.

\bibitem{Kha02}
H.~K. Khalil.
\newblock {\em Nonlinear Systems}.
\newblock Prentice Hall, Upper Saddle River, NJ, USA, third edition, 2002.

\bibitem{KimCWZ22}
Y.~Kim, Y.~Choi, D.~Widemann, and T.~Zohdi.
\newblock A fast and accurate physics-informed neural network reduced order
  model with shallow masked autoencoder.
\newblock {\em J. Comput. Phys.}, 451:110841, 2022.

\bibitem{Kot19}
P.~Kotyczka.
\newblock {\em Numerical Methods for Distributed Parameter Port-Hamiltonian
  Systems}.
\newblock TUM.University Press, Munich, Germany, 2019.

\bibitem{KotL19}
P.~Kotyczka and L.~Lef\`{e}vre.
\newblock Discrete-time port-{H}amiltonian systems: a definition based on
  symplectic integration.
\newblock {\em Systems Control Lett.}, 133:104530, 2019.

\bibitem{KraW19}
B.~Kramer and K.~E. Willcox.
\newblock Nonlinear model order reduction via lifting transformations and
  proper orthogonal decomposition.
\newblock {\em AIAA J.}, 57(6):2297--2307, 2019.

\bibitem{KunM06}
P.~Kunkel and V.~L. Mehrmann.
\newblock {\em Differential-Algebraic Equations -- Analysis and Numerical
  Solution}.
\newblock EMS Publishing House Z\"urich, Switzerland, 2006.

\bibitem{LeeC19}
K.~Lee and K.~T. Carlberg.
\newblock Model reduction of dynamical systems on nonlinear manifolds using
  deep convolutional autoencoders.
\newblock {\em J. Comput. Phys.}, 404:108973, 2019.

\bibitem{MacM04}
A.~Macchelli and C.~Melchiorri.
\newblock Modeling and control of the {T}imoshenko beam. {T}he distributed port
  {H}amiltonian approach.
\newblock {\em SIAM J. Control Optim.}, 43(2):743--767, 2004.

\bibitem{ManBBCDKV08}
J.~Mandel, L.~S. Bennethum, J.~D. Beezley, J.~L. Coen, C.~C. Douglas, M.~Kim,
  and A.~Vodacek.
\newblock A wildland fire model with data assimilation.
\newblock {\em Math. Comput. Simul.}, 79(3):584--606, 2008.

\bibitem{MehM19}
V.~Mehrmann and R.~Morandin.
\newblock Structure-preserving discretization for port-{H}amiltonian descriptor
  systems.
\newblock In {\em Proceedings of the 58th IEEE Conference on Decision and
  Control}, pages 6863--6868, Nice, France, 2019.

\bibitem{MehU23}
V.~Mehrmann and B.~Unger.
\newblock Control of port-{H}amiltonian differential-algebraic systems and
  applications.
\newblock {\em Acta Numer., \upshape{in press}}, 2023.

\bibitem{MonTC13}
N.~Monshizadeh, H.~L. Trentelman, and M.~K. Camlibel.
\newblock Stability and synchronization preserving model reduction of
  multi-agent systems.
\newblock {\em Systems Control Lett.}, 62:1--10, 2013.

\bibitem{Moo81}
B.~C. Moore.
\newblock Principal component analysis in linear systems: {c}ontrollability,
  observability, and model reduction.
\newblock {\em IEEE Trans. Automat. Control}, 26(1):17--32, 1981.

\bibitem{MorLMRY21}
L.~A. Mora, Y.~Le~Gorrec, D.~Matignon, H.~Ramirez, and J.~I. Yuz.
\newblock On port-{H}amiltonian formulations of 3-dimensional compressible
  {N}ewtonian fluids.
\newblock {\em Phys. Fluids}, 33:117117, 2021.

\bibitem{NagS04}
R.~Nagel and E.~Sinestrari.
\newblock The {M}iller scheme in semigroup theory.
\newblock {\em Adv. Differ. Equ.}, 9(3--4):387--414, 2004.

\bibitem{Nes04}
Y.~Nesterov.
\newblock {\em Introductory Lectures on Convex Optimization}.
\newblock Springer Science+Business Media, New York, NY, USA, 2004.

\bibitem{OhlR16}
M.~Ohlberger and S.~Rave.
\newblock Reduced basis methods: success, limitations and future challenges.
\newblock In {\em Proceedings of the Conference Algoritmy}, Vysok\'e Tatry,
  Slovakia, 2016.

\bibitem{OrtSCA08}
R.~Ortega, A.~van~der Schaft, F.~Casta\~nos, and A.~Astolfi.
\newblock Control by interconnection and standard passivity-based control of
  port-{H}amiltonian systems.
\newblock {\em IEEE Trans. Automat. Control}, 53(11):2527--2542, 2008.

\bibitem{PolS12}
R.~V. Polyuga and A.~J. van~der Schaft.
\newblock Effort- and flow-constraint reduction methods for structure
  preserving model reduction of port-{Hamiltonian} systems.
\newblock {\em Systems Control Lett.}, 61(3):412--421, 2012.

\bibitem{Pul19}
R.~Pulch.
\newblock Stability-preserving model order reduction for linear stochastic
  {G}alerkin systems.
\newblock {\em J. Math. Ind.}, 9:10, 2019.

\bibitem{QuaMN16}
A.~Quarteroni, A.~Manzoni, and F.~Negri.
\newblock {\em Reduced Basis Methods for Partial Differential Equations}.
\newblock Springer International Publishing, Cham, Switzerland, 2016.

\bibitem{RamMS13}
H.~Ramirez, B.~Maschke, and D.~Sbarbaro.
\newblock Irreversible port-{Hamiltonian} systems: a general formulation of
  irreversible processes with application to the {CSTR}.
\newblock {\em Chem. Eng. Sci.}, 89:223--234, 2013.

\bibitem{RasCSS21}
R.~Rashad, F.~Califano, F.~P. Schuller, and S.~Stramigioli.
\newblock Port-{H}amiltonian modeling of ideal fluid flow: part {I}.
  {F}oundations and kinetic energy.
\newblock {\em J. Geom. Phys.}, 164:104201, 2021.

\bibitem{RasCSS20}
R.~Rashad, F.~Califano, A.~J. van~der Schaft, and S.~Stramigioli.
\newblock Twenty years of distributed port-{H}amiltonian systems: a literature
  review.
\newblock {\em IMA J. Math. Control Inf.}, 37(4):1400--1422, 2020.

\bibitem{RimPM23}
D.~Rim, B.~Peherstorfer, and K.~T. Mandli.
\newblock Manifold approximations via transported subspaces: model reduction
  for transport-dominated problems.
\newblock {\em SIAM J. Sci. Comput.}, 45(1):A170--A199, 2023.

\bibitem{RowM00}
C.~W. Rowley and J.~E. Marsden.
\newblock Reconstruction equations and the {K}arhunen--{L}o{\`e}ve expansion
  for systems with symmetry.
\newblock {\em Phys. D}, 142(1--2):1--19, 2000.

\bibitem{Sat18}
K.~Sato.
\newblock Riemannian optimal model reduction of linear port-{H}amiltonian
  systems.
\newblock {\em Automatica}, 93:428--434, 2018.

\bibitem{SchPFWM21}
M.~Schaller, F.~Philipp, T.~Faulwasser, K.~Worthmann, and B.~Maschke.
\newblock Control of port-{H}amiltonian systems with minimal energy supply.
\newblock {\em Eur. J. Control}, 62:33--40, 2021.

\bibitem{SchVR08}
W.~H.~A. Schilders, H.~A. van~der Vorst, and J.~Rommes.
\newblock {\em Model Order Reduction: {T}heory, Research Aspects and
  Applications}.
\newblock Springer Berlin Heidelberg, Germany, 2008.

\bibitem{SchRM19}
P.~Schulze, J.~Reiss, and V.~Mehrmann.
\newblock Model reduction for a pulsed detonation combuster via shifted proper
  orthogonal decomposition.
\newblock In R.~King, editor, {\em Active Flow and Combustion Control 2018},
  pages 271--286. Springer International Publishing, Cham, Switzerland, 2019.

\bibitem{SchV21}
P.~Schwerdtner and M.~Voigt.
\newblock Adaptive sampling for structure-preserving model order reduction of
  port-{H}amiltonian systems.
\newblock {\em IFAC-PapersOnLine}, 54(19):143--148, 2021.

\bibitem{Sor05}
D.~C. Sorensen.
\newblock Passivity preserving model reduction via interpolation of spectral
  zeros.
\newblock {\em Systems Control Lett.}, 54:347--360, 2005.

\bibitem{Sch17a}
A.~van~der Schaft.
\newblock {\em $L_2$-Gain and Passivity Techniques in Nonlinear Control}.
\newblock Springer International Publishing, Cham, Switzerland, third edition,
  2017.

\bibitem{Sch20}
A.~van~der Schaft.
\newblock Port-{H}amiltonian modeling for control.
\newblock {\em Annu. Rev. Control Robot. Auton. Syst.}, 3:393--416, 2020.

\bibitem{SchJ14}
A.~J. van~der Schaft and D.~Jeltsema.
\newblock Port-{Hamiltonian} systems theory: an introductory overview.
\newblock {\em Found. Trends Syst. Control}, 1(2--3):173--378, 2014.

\bibitem{Vid93}
M.~Vidyasagar.
\newblock {\em Nonlinear Systems Analysis}.
\newblock Prentice-Hall, Englewood Cliffs, NJ, USA, second edition, 1993.

\bibitem{VosS14}
T.~Vo\ss{} and J.~M.~A. Scherpen.
\newblock Port-{H}amiltonian modeling of a nonlinear {T}imoshenko beam with
  piezo actuation.
\newblock {\em SIAM J. Control Optim.}, 52(1):493--519, 2014.

\bibitem{VuLM16}
N.~M.~T. Vu, L.~Lef\`{e}vre, and B.~Maschke.
\newblock A structured control model for the thermo-magneto-hydrodynamics of
  plasmas in tokamaks.
\newblock {\em Math. Comput. Model. Dyn. Syst.}, 22(3):181--206, 2016.

\bibitem{WanMS18}
L.~Wang, B.~Maschke, and A.~van~der Schaft.
\newblock Port-{H}amiltonian modeling of non-isothermal chemical reaction
  networks.
\newblock {\em J. Math. Chem.}, 56:1707--1727, 2018.

\bibitem{WarBST21}
A.~Warsewa, M.~B\"ohm, O.~Sawodny, and C.~Tar\'in.
\newblock A port-{H}amiltonian approach to modeling the structural dynamics of
  complex systems.
\newblock {\em Appl. Math. Model.}, 89:1528--1546, 2021.

\bibitem{WolLEK10}
T.~Wolf, B.~Lohmann, R.~Eid, and P.~Kotyczka.
\newblock Passivity and structure preserving order reduction of linear
  port-{H}amiltonian systems using {K}rylov subspaces.
\newblock {\em Eur. J. Control}, 16(4):401--406, 2010.

\bibitem{ZhoWHLW21}
W.~Zhou, Y.~Wu, H.~Hu, Y.~Li, and Y.~Wang.
\newblock Port-{H}amiltonian modeling and {IDA}-{PBC} control of an
  {IPMC}-actuated flexible beam.
\newblock {\em Actuators}, 10:236, 2021.

\end{thebibliography}

\end{document}